\documentclass[11pt]{article}
\usepackage[top=2.54cm,left=2.54cm,right=2.54cm,bottom=2.54cm]{geometry}
\usepackage{amsmath,amsfonts}
\usepackage{algorithmic}
\usepackage{algorithm}
\usepackage{array}
\usepackage[caption=false,font=normalsize,labelfont=sf,textfont=sf]{subfig}
\usepackage{textcomp}
\usepackage{stfloats}
\usepackage{url}
\usepackage{verbatim}
\usepackage{graphicx}
\usepackage{cite}
\usepackage{tablefootnote}
\usepackage{xcolor}
\usepackage{hyperref}

\DeclareMathOperator*{\argmax}{arg\,max}
\DeclareMathOperator*{\argmin}{arg\,min}
\usepackage{multirow}
\usepackage[normalem]{ulem}
\useunder{\uline}{\ul}{}
\usepackage{amsthm}
\newtheorem{defin}{{Definition}}
\newtheorem{assumption}{{Assumption}}
\newtheorem{theorem}{Theorem}
\newtheorem{lemma}{Lemma} 
\newtheorem{remark}{Remark}

\newcommand{\be}{\begin{equation}}
\newcommand{\ee}{\end{equation}}

\newcommand{\fail}{\textrm{fail}}
\providecommand{\keywords}[1]
{
  \small	
  \textbf{\textit{Keywords---}} #1
}

\newcommand{\tcb}[1]{{#1}}
\newcommand{\tcr}[1]{{#1}}
\begin{document}

\title{A New Inexact Proximal Linear Algorithm with Adaptive Stopping Criteria for Robust Phase Retrieval}

\author{Zhong Zheng, Shiqian Ma, and Lingzhou Xue
\footnote{{Z. Zheng} and {L. Xue} are with the Department of Statistics at Penn State University, University Park, PA 16802.  Research is partly supported by NSF grants DMS-1953189, CCF-2007823, and DMS-2210775. Email: \{zvz5337,
lzxue\}@psu.edu. {S. Ma} is with the Department of Computational Applied Mathematics and Operations Research at Rice University, Houston, TX 77005. Research is supported in part by NSF grants DMS-2243650, CCF-2308597, CCF-2311275, and ECCS-2326591, and a startup fund from Rice University. Email: sqma@rice.edu. {S. Ma} and {L. Xue} are co-corresponding authors. } 
}



\date{}

\maketitle

\begin{abstract}
This paper considers the robust phase retrieval problem, which can be cast as a nonsmooth and nonconvex optimization problem. We propose a new inexact proximal linear algorithm with the subproblem being solved inexactly. Our contributions are two adaptive stopping criteria for the subproblem. The convergence behavior of the proposed methods is analyzed. Through experiments on both synthetic and real datasets, we demonstrate that our methods are much more efficient than existing methods, such as the original proximal linear algorithm and the subgradient method. 
\end{abstract}


\section{Introduction}
Phase retrieval aims to recover a signal from intensity-based or magnitude-based measurements. It finds various applications in different fields, including X-ray crystallography \cite{miao1999extending}, optics \cite{millane1990phase}, array and high-power coherent diffractive imaging \cite{chai2010array}, astronomy \cite{fienup1987phase} and microscopy \cite{miao2008extending}. Mathematically, phase retrieval tries to find the true signal vectors $x_\star$ or $-x_\star$ in $\mathbb{R}^n$ from a set of magnitude measurements:
\begin{equation}\label{intro_noiseless_phase}
   b_{i}=    (a_i^\top x_\star)^2, \text{ for } i=1,2,\ldots,m,
\end{equation}
where $a_i\in \mathbb{R}^n$ and $b_i\geq 0, i=1,2,\ldots,m$. 
Directly solving the equations leads to an NP-hard problem \cite{fickus2014phase}, and nonconvex algorithms based on different designs of objective functions have been well studied in the literature, including Wirtinger flow \cite{candes2015phase}, truncated Wirtinger flow \cite{chen2017solving}, truncated amplitude flow \cite{wang2017solving}, \tcb{reshaped Wirtinger flow \cite{zhang2017nonconvex}}, etc.

In this paper, we focus on the robust phase retrieval (RPR) problem \cite{duchi2019solving}, which considers the case where $b_i$ contains noise due to measurement errors of equipment. That is, 
\begin{equation}\label{intro_corrupted_phase}
    b_{i} = \begin{cases}
(a_i^\top x_\star)^2, & i\in \mathbb{I}_1,\\
\xi_i, & i\in \mathbb{I}_2,
\end{cases}
\end{equation}
in which $\mathbb{I}_1\bigcup \mathbb{I}_2 = \{1,2\ldots,m\}$, $\mathbb{I}_1\cap\mathbb{I}_2 = \emptyset$, and $\xi_i$ denotes a random noise. \cite{duchi2019solving} proposed to formulate RPR as the following optimization problem:
\begin{equation}\label{duchi_l1_ori}
    \min_{x\in \mathbb{R}^{n}} F(x):=\frac{1}{m}\sum_{i=1}^{m}\left|(a_i^\top x)^2-b_i\right|.
\end{equation}
It is demonstrated in \cite{duchi2019solving} that using \eqref{duchi_l1_ori} for RPR possesses better recoverability compared to the median truncated Wirtinger flow algorithm \cite{zhang2016provable} based on the $\ell_2$-loss. 

Solving \eqref{duchi_l1_ori} is challenging because it is a nonconvex and nonsmooth optimization problem. 
In 
\cite{davis2018subgradient}, the authors proposed the subgradient method to solve it. This method requires geometrically decaying step size, and it is unclear how to schedule this kind of step size in practice.  \cite{duchi2019solving} proposed to use the proximal linear (PL) algorithm to solve \eqref{duchi_l1_ori}. For ease of presentation, we rewrite \eqref{duchi_l1_ori} as 
\be\label{RPR}
\min_{x\in \mathbb{R}^{n}} F(x) = h(c(x)) = \frac{1}{m}\left\||Ax|^2-b\right\|_1,
\ee
where $Ax = [\langle a_1, x\rangle, \ldots, \langle a_m, x\rangle]^\top$, $b=[b_1,\ldots,b_m]^\top$, $h(z) := \frac{1}{m}\|z\|_1$, and $c(x) := |Ax|^2-b$ is a smooth map in which $|\cdot|^2$ is element-wise square. One typical iteration of the PL algorithm is 
\be\label{PL}
x^{k+1} \approx \argmin_{x\in \mathbb{R}^{n}}\ F_t(x;x^k),
\ee
where $t>0$ is the step size,
\begin{equation}\label{def-F_l}
    F(z;y) := h(c(y) + \nabla c(y)(z-y)),
\end{equation}
\begin{equation}\label{def-Ft}
    F_t(z;y) := F(z;y) + \frac{1}{2t}\|z-y\|_2^2,
\end{equation}
$\nabla c$ denotes the Jacobian of $c$, and ``$\approx$" means that the subproblem is solved inexactly. The subproblem \eqref{PL} is convex and can be solved by various methods such as the proximal operator graph splitting (POGS) algorithm used in \cite{duchi2019solving}. The PL method has drawn lots of attention recently. It has been studied by \cite{lewis2016proximal,drusvyatskiy2019efficiency,duchi2018stochastic} and applied to solving many important applications such as RPR \cite{duchi2019solving}, robust matrix recovery \cite{charisopoulos2021low,charisopoulos2019composite}, and sparse spectral clustering \cite{wang2020manifold}. The subproblem \eqref{PL} is usually solved inexactly for practical concerns. As pointed out in \cite{duchi2019solving}, the PL implemented in \cite{duchi2019solving} is much slower than the median truncated Wirtinger flow method. We found that this is mainly due to their stopping criterion for solving the subproblem \eqref{PL}, which unnecessarily solves \eqref{PL} to very high accuracy in the early stage of the algorithm. Moreover, we found that the POGS algorithm used in \cite{duchi2019solving} is ineffective in solving the subproblem \eqref{PL}. 
In this paper, we propose adaptive stopping criteria for inexactly solving \eqref{PL} with the fast iterative shrinkage-thresholding algorithm (FISTA) \cite{nesterov1988,beck2009fast,nesterov2013gradient}. We found that our new inexact PL (IPL) with the adaptive stopping criteria greatly outperforms existing implementations of PL methods \cite{duchi2019solving} for solving RPR \eqref{RPR}. 

{\bf Our Contributions.} In this paper, we propose two new adaptive stopping criteria for inexactly solving \eqref{PL}. The first one ensures that \eqref{PL} is solved to a relatively low accuracy: 
\begin{equation}\label{low0}
\begin{aligned}
    \text{(LACC)} &\quad F_t(x^{k+1};x^k)-\min_{x\in \mathbb{R}^n}\ F_t(x;x^k) \\
      & \leq\rho_l\left(F_t(x^{k};x^k) - F_t(x^{k+1};x^k)\right), \quad \rho_l > 0,
\end{aligned}
\end{equation}
and the second one ensures that \eqref{PL} is solved to a relatively high accuracy:
\begin{equation}\label{high0}
\begin{aligned}
    \text{(HACC)} &\quad F_t(x^{k+1};x^k)-\min_{x\in \mathbb{R}^n}\ F_t(x;x^k) \\
      & \leq\frac{\rho_h}{2t}\|x^{k+1} - x^k\|_2^2, \quad 0 <\rho_h<1/4.
\end{aligned}
\end{equation}
Here, $\rho_l$ and $\rho_h$ are given constants. \tcb{Similar to the proximal bundle method \cite{diaz2023optimal} for nonsmooth convex problems,  (LACC) and (HACC) are designed to ensure the sufficient \tcr{decrease} of the objective function for the nonsmooth and nonconvex RPR problem.} Note that both (LACC) and (HACC) are only used theoretically because $\min_{x\in \mathbb{R}^n}\ F_t(x;x^k)$ is not available. Later we will propose more practical stopping criteria that can guarantee (LACC) and (HACC). The connections of our approach to existing work are listed below. 

\begin{itemize}
\item[(a)] Our (LACC) condition coincides with the inexact stopping criterion proposed in \cite{bonettini2016variable,lee2019inexact,lee2020inexact} for the proximal gradient method. In these papers, the authors focus on a different optimization problem 
\begin{equation*}
    \min_{x\in \mathbb{R}^n} f_0(x) := f_1(x) + f_2(x),
\end{equation*}
in which $f_1$ is a smooth function, and $f_2$ is a proper, convex, and lower semi-continuous function. One typical iteration of their algorithms can be written as 
\begin{subequations}
\begin{align}
y^{k+1} \approx & \ \min_{x\in\mathbb{R}^n} f_{0k}(x) = f_1(x^k) + (x-x_k)^\top\nabla f_1(x^k) \nonumber\\
&+ f_2(x) + \frac{1}{2}(x-x^k)^\top H_k(x-x^k), \label{Bonettini-subproblem}\\
x^{k+1} = & \ x^k + \lambda_k(y^{k+1} - x^k), \label{Bonettini-subproblem-2}
\end{align}
\end{subequations}
where $H_k\in\mathbb{R}^{n\times n}$ is a positive semi-definite matrix and $\lambda_k\in[0,1]$ is a step size. 
The stopping criterion for inexactly solving \eqref{Bonettini-subproblem} proposed in \cite{bonettini2016variable,lee2019inexact,lee2020inexact} is
\begin{equation}\label{Bonettini-subproblem-stop-crit} f_{0k}(y^{k+1}) - \Tilde{f}_{0k}\leq \eta\left(f_{0k}(x^{k}) - \Tilde{f}_{0k}\right),
\end{equation}
where $\Tilde{f}_{0k} = \min_{x\in\mathbb{R}^n} f_{0k}(x)$ and $\eta\in(0,1)$. We note that this is the same as our (LACC). Therefore, our (LACC) is essentially an extension of \eqref{Bonettini-subproblem-stop-crit} from the proximal gradient method to the proximal linear method. 
\item[(b)] To the best of our knowledge, our (HACC) criterion is new and serves as a good alternative to (LACC). From our numerical experiments, we found that (HACC) works comparably with (LACC), and we believe that it can be useful for other applications.
\item[(c)] We analyze the overall complexity and the local convergence of our IPL algorithm for solving RPR under the sharpness condition. To the best of our knowledge, this is the first time such results have been obtained under the sharpness condition.

\item[(d)] We propose to solve \eqref{PL} inexactly using FISTA \cite{nesterov1988,beck2009fast,nesterov2013gradient}, which uses easily verifiable stopping conditions that can guarantee (LACC) and (HACC). Through extensive numerical experiments, we demonstrate that our IPL with the new stopping criteria significantly outperforms existing algorithms for solving RPR. 
\end{itemize}

{\bf Organization.} The rest of this paper is organized as follows. In Section \ref{sec:main-alg-and-convergence}, we propose the main framework of our inexact proximal linear algorithm with two new adaptive stopping criteria for the subproblem. We establish its iteration complexity for obtaining an $\epsilon$-stationary point and its local convergence under the sharpness condition. Connections with some existing methods are also discussed. In Section \ref{sec:fista}, we discuss how to adapt the FISTA to solve the subproblem inexactly. We also establish the overall complexity of FISTA -- the total number of iterations of the FISTA -- in order to obtain an $\epsilon$-optimal solution under the sharpness condition. In Section \ref{sec:numerical}, we show the numerical results on both synthetic and real datasets to demonstrate the advantage of the proposed methods over some existing methods. The proofs for all the theorems and lemmas are given in Section \ref{sec:proof}. Finally, we include some concluding remarks in Section \ref{sec:conclusion}.

\section{IPL and Its Convergence Analysis}\label{sec:main-alg-and-convergence}
In this section, we introduce our IPL algorithm for solving the RPR \eqref{RPR} with the inexact stopping criteria (LACC) and (HACC) for the subproblem \eqref{PL} and analyze its convergence. We will discuss the FISTA for solving \eqref{PL} that guarantees (LACC) and (HACC) in the next section. 

We first follow \cite{drusvyatskiy2019efficiency} to introduce some notation. Let
$$S_t(y) := \argmin_{x\in \mathbb{R}^n} \ F_t(x;y),$$
$$\varepsilon(x;y) := F_t(x;y) - F_t(S_t(y);y),$$
where $F_t(z;y)$ is defined in \eqref{def-Ft}. \tcb{We will also use the notation
$$L = \frac{2}{m}\|A\|_2^2 = \frac{2}{m}\left\|\sum_{i=1}^m a_ia_i^\top\right\|_2.$$}
A \tcb{Meta Algorithm} of our IPL is summarized in Algorithm \ref{alg:adaptive-IPL}. We again emphasize that Algorithm \ref{alg:adaptive-IPL} cannot be implemented because $\min_{x\in \mathbb{R}^n}\ F_t(x;x^k)$ is not available, and we will discuss practical versions of it in the next section. 
\begin{algorithm}[ht]  
\caption{IPL -- A \tcb{Meta Algorithm}}
\label{alg:adaptive-IPL}
\begin{algorithmic}
\STATE \textbf{Input:} Initial point $x^0$, step size $t = 1/L$, 
parameters $\rho_l>0$ and $\rho_h\in (0,1/4)$
\FOR{$k = 0, 1, \ldots, $}
\STATE{Obtain $x^{k+1}$ by inexactly solving \eqref{PL} with one of the following stopping criteria:}
\STATE{\quad\textbf{Option 1:} (LACC), i.e., \eqref{low0}}
\STATE{\quad\textbf{Option 2:} (HACC), i.e., \eqref{high0}}
\ENDFOR
\end{algorithmic}
\end{algorithm}
\subsection{Convergence under General Settings}
In this subsection, we analyze the convergence rate of IPL (Algorithm \ref{alg:adaptive-IPL}) for obtaining an $\epsilon$-stationary point of \eqref{RPR} \tcb{under the general settings when the sharpness condition may not hold.} We use the definition of $\epsilon$-stationary point as introduced in \cite{drusvyatskiy2019efficiency}. 


\begin{defin}\label{def:epsilon_stationary_point}
We call $\tilde{x}$ an $\epsilon$-stationary point of \eqref{RPR} if the following inequality holds: 
\begin{equation}\label{cond:gen_exact_termi}
\|\mathcal{G}_t(\tilde{x})\|_2\leq \epsilon,
\end{equation}
where $\mathcal{G}_t(x)$ is the proximal gradient which is defined as:
\begin{equation}\label{def:gen_proxgradient}
\mathcal{G}_t(x)=t^{-1}\left(x-S_t(x)\right).
\end{equation}
\end{defin}

Our convergence rate result of Algorithm \ref{alg:adaptive-IPL} is given in Theorem \ref{thm:gen_convergence_rate}, and the proof is given in Section \ref{sec:proof}. 

\begin{theorem}\label{thm:gen_convergence_rate}
Denote $F^\star = \inf_{x\in\mathbb{R}^n} F(x)$. For Algorithm \ref{alg:adaptive-IPL} with $t = 1/L$, the following conclusion holds.
\begin{itemize}
    \item[(a)] When (LACC) holds \tcb{with $\rho_l>0$} for any $k\in\mathbb{N}$, we can find an $\epsilon$-stationary point in \tcb{$$\left\lfloor\frac{2(1+\rho_l)(F(x^0) - F^\star)}{t\epsilon^2}\right\rfloor$$}
    iterations for any $\epsilon>0$.
    \item[(b)] When (HACC) holds \tcb{with $0<\rho_h<1/4$} for any $k\in\mathbb{N}$, we can find an $\epsilon$-stationary point in \tcb{$$\left\lfloor\frac{2(1-\sqrt{\rho_h})^2(F(x^0) - F^\star)}{(1-2\sqrt{\rho_h})t\epsilon^2}\right\rfloor$$}
    iterations for any $\epsilon>0$.
\end{itemize}
\end{theorem}
\tcb{Theorem 1 shows that IPL finds an $\epsilon$-stationary point in $O(1/\epsilon^2)$ main iterations with the adaptive IPL stopping conditions. Moreover, Theorem 1  achieves the best known convergence rate for PL in  \cite{drusvyatskiy2019efficiency}. We should point out that we use two adaptive stopping criteria for the subproblem, but  \cite{drusvyatskiy2019efficiency} requires solving the subproblem \eqref{PL} exactly (see their Proposition 3) or using their pre-determined subproblem accuracy conditions (see their Theorem 5.2).}

\subsection{Local Convergence under Sharpness Assumption}

In this subsection, we analyze the local convergence of IPL (Algorithm \ref{alg:adaptive-IPL}) to the global optimal solution under the sharpness condition. 

\begin{assumption}[Sharpness]\label{ass:sharpness}
There exists a constant $\lambda_s>0$ such that the following inequality holds for any $x\in \mathbb{R}^n$:
\begin{equation}\label{ineq:sharpness}
F(x)-F(x_\star)\geq \lambda_s\Delta(x),
\end{equation}
where $\Delta(x) := \min\{\|x-x_\star\|_2,\|x+x_\star\|_2\}$.
\end{assumption}

\cite{duchi2019solving} proved that the sharpness condition (Assumption \ref{ass:sharpness}) is satisfied by the RPR \eqref{RPR} with high probability under certain mild conditions.

Another assumption is about the closeness between the initial point and the optimal solution, which can be guaranteed by the modified spectral initialization (see Algorithm 3 in \cite{duchi2019solving}) with high probability under some mild conditions.

\begin{assumption}\label{ass:init}
Under Assumption \ref{ass:sharpness}, we assume that the initial point $x^0$ in Algorithm \ref{alg:adaptive-IPL} satisfies the following inequalities. 
\begin{enumerate}
\item[(a)] If (LACC) is chosen in Algorithm \ref{alg:adaptive-IPL}, then we assume $x^0$ satisfies
\begin{equation}\label{cond:init_low}
F(x^0)-F(x_\star)\leq \lambda_s^2/(2L).
\end{equation}
\item[(b)] If (HACC) is chosen in Algorithm \ref{alg:adaptive-IPL}, then we assume $x^0$ satisfies
\begin{equation}\label{cond:init_high}
\Delta(x^0)\leq \frac{\lambda_s(1-4\rho_h)}{2(1-3\rho_h)L}.
\end{equation}
\end{enumerate}
\end{assumption}

We now define the $\epsilon$-optimal solution to RPR \eqref{RPR}. 
\begin{defin}
    We call $\bar{x}$ an $\epsilon$-optimal solution to RPR \eqref{RPR}, if $\Delta(\bar{x})\leq \epsilon$.
\end{defin}

Now we are ready to show in Theorem \ref{thm:convergence_rate_sharpness} that, in terms of main iteration number, (LACC) leads to local linear convergence and (HACC) leads to local quadratic convergence.

\begin{theorem}\label{thm:convergence_rate_sharpness}
Let $t=\frac{1}{L}$ and suppose that Assumption \ref{ass:sharpness} holds. For the sequence $\{x^k\}_{k=0}^\infty$ generated by Algorithm \ref{alg:adaptive-IPL}, we have the following conclusions.
\begin{enumerate}
\item[(a)] (Low Accuracy) When \eqref{cond:init_low} holds and \eqref{low0} holds \tcb{with $\rho_l>0$} for any $k\in \mathbb{N}$, we have
$$\Delta(x^k)\leq \frac{F(x^0)-F(x_\star)}{\lambda_s}\left(\frac{1+4\rho_l}{2+4\rho_l}\right)^k, \qquad \forall k\in \mathbb{N}.$$
\item[(b)] (High Accuracy) When \eqref{cond:init_high} holds and \eqref{high0} holds \tcb{with $0<\rho_h<1/4$} for any $k\in \mathbb{N}$, we have
$$\Delta(x^k)\leq \frac{\lambda_s(1-4\rho_h)}{L(1-3\rho_h)} \zeta^{2^k}, \qquad \forall k\in \mathbb{N},$$
where 
$\zeta := \frac{L\Delta(x^0)(1-3\rho_h)}{\lambda_s(1-4\rho_h)}$. 
\end{enumerate}
\end{theorem}

Theorem \ref{thm:convergence_rate_sharpness} shows that, with a good initialization, using (LACC) finds an $\epsilon$-optimal solution to \eqref{RPR} within $O(\log \frac{1}{\epsilon})$ iterations, which is a linear rate, and using (HACC) finds an $\epsilon$-optimal solution to \eqref{RPR} within $O(\log\log \frac{1}{\epsilon})$ iterations, which is a quadratic rate.

\subsection{Related Work}

There are two closely related works that need to be discussed here. \cite{duchi2019solving} studied the PL algorithm for solving RPR \eqref{RPR}, and established its local quadratic convergence under the sharpness condition. But their theoretical analysis requires the subproblem \eqref{PL} to be solved exactly. In practice, \cite{duchi2019solving} proposed to use POGS \tcb{\cite{parikh2014block}}, which is a variant of the alternating direction method of multipliers (ADMM), to solve \eqref{PL} inexactly. However, they did not provide any convergence analysis for the algorithm when the subproblem \eqref{PL} is solved inexactly by POGS. 
\cite{drusvyatskiy2019efficiency} also considered solving \eqref{RPR} for obtaining an $\epsilon$-stationary point as defined in Definition \ref{def:epsilon_stationary_point}.\footnote{The authors of \cite{drusvyatskiy2019efficiency} actually considered solving a more general problem $\min_{x} g(x)+h(c(x))$. Here, for simplicity, we assume that $g=0$ and this does not affect the discussion.} Indeed, several algorithms were proposed and analyzed in \cite{drusvyatskiy2019efficiency}. In particular, a practical algorithm proposed by \cite{drusvyatskiy2019efficiency} uses FISTA \cite{nesterov1988,beck2009fast,nesterov2013gradient} to inexactly solve

\begin{align*}
\min_{\lambda\in\mathbb{R}^m}\phi_{k,\nu}(\lambda) = &\frac{t}{2}\|x^k/t-\nabla c(x^k)^\top\lambda\|_2^2\\
    &- \lambda^\top(c(x^k) - \nabla c(x^k)x^k) +(h_\nu)^\star(\lambda),
\end{align*}
which is the dual problem of a smoothed version of \eqref{PL}. Here $(h_\nu)^\star(\cdot)$ is the Fenchel conjugate of $h_\nu$, and $h_\nu$ is the Moreau envelope of $h$, which is defined as  
$$h_\nu(\lambda) = \inf_{\lambda'\in\mathbb{R}^m} h(\lambda')+ \frac{1}{2\nu}\|\lambda'-\lambda\|_2^2.$$
\cite{drusvyatskiy2019efficiency} proposed to terminate the FISTA when $\mbox{dist}(0;\partial \phi_{k,\nu}(\lambda^{k+1}))\leq \frac{1}{L_h(k+1)^2}$, and then update $x^k$ by
$x^{k+1} = x^k - t\nabla c(x^k)^\top\lambda^{k+1}$. The authors established the overall complexity of this algorithm for suitably chosen parameters $t$ and $\nu$. Compared to \cite{drusvyatskiy2019efficiency}, we use adaptive stopping criteria and provide a better convergence rate based on Assumption \ref{ass:sharpness}.

\section{FISTA for Solving the Subproblem Inexactly}\label{sec:fista}

In this section, we propose to use the FISTA to inexactly solve \eqref{PL} with more practical stopping criteria that guarantee (LACC) and (HACC). Therefore, the convergence results (Theorems \ref{thm:gen_convergence_rate} and \ref{thm:convergence_rate_sharpness}) in Section \ref{sec:main-alg-and-convergence}
still apply here.

For simplicity, we let $t = 1/L$ throughout this section and rewrite \eqref{PL} as follows.
\be\label{PL-rewrite-2}
\min_{z\in \mathbb{R}^n} \ H_k(z) = \frac{1}{2t}\|z\|_2^2 + \|B_kz-d_k\|_1,
\ee
where we denote
$z = x -x^k,$ $B_k = \frac{2}{m}\mbox{diag}(Ax^k)A,$ and $d_k = \frac{1}{m}\left(b - (Ax^k)^2\right)$.  
As a result, (LACC) and (HACC) can be rewritten respectively as 
\begin{equation}\label{low1}
	H_k(z_{k}) - \min_{z\in \mathbb{R}^n} H_k(z)\leq \rho_l\left(H_k(0) - H_k(z_{k})\right),\rho_l\geq 0, 
\end{equation}
and
\begin{equation}\label{high1}
	H_k(z_{k}) - \min_{z\in \mathbb{R}^n} H_k(z)\leq \frac{\rho_h}{2t}\|z_k\|_2^2,0\leq\rho_h<1/4.
\end{equation}
In IPL, we set $x^{k+1} = x^k + z_k$, where $z_k$ satisfies either \eqref{low1} or \eqref{high1}. The dual problem of \eqref{PL-rewrite-2} is
\begin{equation}\label{sub_dual}
	\max_{\lambda\in R^m,\|\lambda\|_\infty\leq 1} D_k(\lambda) = -\frac{t}{2}\left\|B_k^\top\lambda\right\|_2^2-\lambda^\top d_k.
\end{equation}
From weak duality, we know that
\begin{align}\label{weak-duality}
D_k(\lambda) \leq H_k(z), \forall z\in\mathbb{R}^n, \mbox{ and } \|\lambda\|_\infty\leq 1.
\end{align}
Therefore, $D_k(\lambda)$ can serve as a lower bound for $\min_z H_k(z)$, and we can obtain verifiable stopping criteria that are sufficient conditions for \eqref{low1} and \eqref{high1}. This leads to our inexact FISTA for solving \eqref{PL-rewrite-2}, which is summarized in Algorithm \ref{Nesterov}. Here we define $z_k(\lambda) = -tB_k^\top \lambda$. 

\begin{algorithm}[ht]
	\caption{FISTA for Solving \eqref{sub_dual}}
	\label{Nesterov}
	\begin{algorithmic}
		\STATE {\bf Input:} $\lambda^0\in\mathbb{R}^m$ satisfying $\|\lambda^0\|_\infty\leq 1$. $\lambda_a^0=\lambda_b^0=\lambda_c^0 = \lambda^0$, $\gamma_0=1$, $\rho_l>0$ and $\rho_h\in (0,1/4)$.
		\FOR{$j = 0,1,2\ldots$}
		\STATE \begin{align}
	\lambda_c^{j+1} & = (1-\gamma_j)\lambda_a^j+\gamma_j\lambda_b^j, \nonumber\\ 
	\lambda_b^{j+1} & = \argmin_{\|\lambda\|_\infty\leq 1} \ \frac{{\gamma_j}}{2t_{kj}}\|\lambda - \lambda_c^{j+1}\|^2_2 + \nonumber\\
	&(\lambda - \lambda_c^{j+1})^\top\left(tB_{k}B_k^\top\lambda_c^{j+1} + d_k\right),t_{kj}>0,\label{inner_ls} \\ 
	\lambda_a^{j+1} & = (1-\gamma_j)\lambda_a^j+\gamma_j\lambda_b^{j+1},\nonumber\\
	\gamma_{j+1} & =2/\left(1+\sqrt{1+4/\gamma_j^2}\right),\nonumber
\end{align}
\STATE Terminate if one of the following stopping criteria is satisfied:
	\begin{align} 
		\text{(LACC-FISTA)} &\quad  H_k(z_k(\lambda_a^{j+1})) - D_k(\lambda_a^{j+1})) \leq \nonumber\\
		&\rho_l(H_k(0) - H_k(z_k(\lambda_a^{j+1}))),\label{low2} \\ 
		\text{(HACC-FISTA)} &\quad  H_k(z_k(\lambda_a^{j+1})) - D_k(\lambda_a^{j+1})) \leq \nonumber\\
		& \frac{\rho_h}{2t}\|z_k(\lambda_a^{j+1})\|_2^2.\label{high2}
	\end{align}
		\ENDFOR
		\STATE {\bf Output:} $\lambda_k = \lambda_a^{j+1}, z_{k} =  -tB_{k}^\top\lambda_k$, $x^{k+1} = x^k + z_{k}$.
	\end{algorithmic}
\end{algorithm}

\begin{remark}\label{remark-alg-nesterov}
	Here we remark on the step size $t_{kj}$ in \eqref{inner_ls}. It can be chosen as $\left(tL_k^2\right)^{-1}$ for some $L_k\geq \|B_k\|_2$ or chosen by the Armijo backtracking line search. More specifically, suppose that we have an initial step size $t_{k(-1)}>0$. Given the step size $t_{k(j-1)},j\geq 0$, denote 
 {	\begin{align*}
		\lambda_{b,t_s}^{j+1} = \argmin_{\lambda\in\mathbb{R}^m,\|\lambda\|_\infty\leq 1}& \frac{\gamma_j}{2t_{s}}\|\lambda - \lambda_c^{j+1}\|^2_2\\
  &+ (\lambda - \lambda_c^{j+1})^\top\left(tB_{k}B_k^\top\lambda_c^{j+1} + d_k\right),
	\end{align*}}
 and
{ $$\lambda_{a,t_s}^{j+1} = (1-\gamma_j)\lambda_a^j+\gamma_j\lambda_{b,t_s}^{j+1}$$
	$t_{kj}$ can be selected as
	\begin{align*}
		t_{kj} = &\max\{t_s|t_s = 2^{-s}t_{k(j-1)},s\in\mathbb{N},\\
  &\frac{1}{2t_s}\|\lambda_{a,t_s}^{j+1}-\lambda_c^{j+1}\|_2^2\geq \frac{t}{2}\|B_k^\top\lambda_{a,t_s}^{j+1}-B_k^\top\lambda_c^{j+1}\|_2^2\}.
	\end{align*}}
\end{remark}

We now discuss the overall complexity of the IPL (Algorithm \ref{alg:adaptive-IPL}) with the subproblem \eqref{PL} solved inexactly by FISTA (Algorithm \ref{Nesterov}). For ease of presentation, we denote this algorithm as IPL+FISTA. We assume that IPL is terminated after $K_\epsilon$ iterations, when an $\epsilon$-optimal solution is found, i.e., $\Delta(x^{K_\epsilon})\leq \epsilon,\Delta(x^{K_\epsilon-1})>\epsilon$. We use $J_k,k\geq 0$ to denote the number of iterations of FISTA when it is called in the $k$-th iteration (getting $x^{k+1}$ from $x^k$) of IPL. The overall complexity of IPL+FISTA for obtaining an $\epsilon$-optimal solution is thus given by $J(\epsilon) = \sum_{k=0}^{K_\epsilon - 1}J_k$, which equals the times that we call \eqref{inner_ls}. 
Now we are ready to give the overall complexity of IPL+FISTA in Theorem \ref{overall_iplnesterov}.
\begin{theorem}\label{overall_iplnesterov}
	Let $t = 1/L$, $t_{kj} = \left(t\|B_k\|_2^2\right)^{-1}$ and suppose that Assumption \ref{ass:sharpness} holds. 
	\begin{itemize}
		\item[(a)] (Low Accuracy) For the overall complexity of Algorithm \ref{alg:adaptive-IPL}, when using Algorithm \ref{Nesterov} with option \eqref{low2} \tcb{with $\rho_l>0$}, for any $x^0\in\mathbb{R}^n$ that satisfies $\Delta(x^0)\leq E_1$, we have
		\begin{align}\label{thm3-conclusion-a}
  J(\epsilon)\leq E_2/\epsilon,\forall \epsilon>0.
        \end{align}
		Here $E_1,E_2$ are positive constants that only depend on 
  $\{a_i\}_{i=1}^m,\{b_i\}_{i=1}^m,x_\star,\lambda_s,L$ and $\rho_l$, and they will be specified later in the proof.
		\item[(b)] (High Accuracy) For the overall complexity of Algorithm \ref{alg:adaptive-IPL}, when using Algorithm \ref{Nesterov} with option \eqref{high2} \tcb{with $\rho_h\in (0,1/4)$}, there exist positive constants $E_3,E_4$, and $\{\epsilon_i\}_{i=1}^\infty$ with $\epsilon_i>\epsilon_{i+1},\forall i\in\mathbb{N}_+$ and $\lim_{i\rightarrow \infty}\epsilon_i = 0$ such that 
  if $\Delta(x^0)\leq E_3$, then 
		\begin{align}\label{thm3-conclusion-b}
        J(\epsilon_i)\leq E_4/\epsilon_i,\forall i\in\mathbb{N}_+.
        \end{align}
  Here $E_3,E_4$ only depend on 
$\{a_i\}_{i=1}^m,\{b_i\}_{i=1}^m,x_\star,\lambda_s,L$ and $\rho_h$, $\{\epsilon_i\}_{i=1}^\infty$ depends on $\{\Delta(x^i)\}_{i=1}^\infty$ and they will be specified later in the proof. \tcb{Note that \eqref{thm3-conclusion-b} implies that the worst-case overall complexity might be higher than $O(1/\epsilon)$ -- see the explanation below for more details.}
	\end{itemize}
\end{theorem}

Theorem \ref{overall_iplnesterov} shows that under the (LACC-FISTA) stopping criterion, we need $O(1/\epsilon)$ iterations to find an $\epsilon$-optimal solution, and under the (HACC-FISTA) stopping criterion, we have the same rate with regard to a countable positive sequence that decreases to zero. Theorem \ref{overall_iplnesterov} provides better theoretical rates compared to $O(1/\epsilon^3)$ in \cite{drusvyatskiy2019efficiency}. Moreover, the results in Theorem \ref{overall_iplnesterov} are about the convergence to $\epsilon$-optimal solution, while the results in \cite{drusvyatskiy2019efficiency} are for convergence to $\epsilon$-stationary point. Our results require the sharpness condition, which was not assumed in \cite{drusvyatskiy2019efficiency}.

\tcb{For (b) in Theorem \ref{overall_iplnesterov}, we can only find a countable sequence of diminishing $\epsilon_i$'s to show the $O(1/\epsilon_i)$ rate. We cannot show the $O(1/\epsilon)$ rate for any fixed $\epsilon>0$. This is because of the local quadratic convergence under (HACC) shown in (b) of Theorem \ref{overall_iplnesterov}. For instance, if our initial point $x^0$ satisfies $\Delta(x^0) = 2\epsilon >\epsilon$, under the (HACC), the quadratic convergence result in  Theorem \ref{overall_iplnesterov} (b) implies that $\Delta(x^{1})\leq C\epsilon^2 <\epsilon$ for some constant $C>0$ when $\epsilon$ is sufficiently small. Therefore, IPL finds an $\epsilon$-optimal stationary point with only one main iteration. However, Lemma \ref{subiters_nesterov} (b) indicates that $J(\epsilon) = J_0 = O(1/\epsilon^2)$. Therefore, it is possible that the overall complexity becomes $O(1/\epsilon^2)$, which is higher than $O(1/\epsilon)$. 
}

\section{Numerical Experiments}\label{sec:numerical}

In this section, we conduct numerical experiments to compare our IPL method with existing methods for solving the RPR problem \eqref{RPR}. \tcb{Readers can find the code and datasets to replicate the experiments in this section via \url{https://github.com/zhengzhongpku/IPL-code-share}.} The algorithms that we test include the following ones. 
\begin{itemize}
	\item[(i)] \textbf{PL}: The original proximal linear algorithm proposed by \cite{duchi2019solving} where the subproblem \eqref{PL} is solved by POGS \cite{parikh2014block}. POGS terminates when both the primal residual and the dual residual are small enough. In their code, the authors \cite{duchi2019solving} implemented a two-stage trick that uses a relatively larger tolerance in early iterations and a smaller tolerance in later iterations to terminate the POGS. In our comparison, we use all the default parameters set by the authors in their \tcb{code}\footnote{The code of \cite{duchi2019solving} can be downloaded from \url{https://web.stanford.edu/~jduchi/projects/phase-retrieval-code.tgz}}.
	\item[(ii)] \textbf{Subgradient method}. The subgradient method with geometrically decaying step sizes was proposed by \cite{davis2018subgradient}, and they used this algorithm to solve the RPR \eqref{RPR}. One typical iteration of this algorithm is 
	\begin{align}\label{subgradient-method}
    x^{k+1} = x^k - \lambda_0 q^k \xi_k/\|\xi_k\|_2,k\geq 0,
    \end{align}
	in which $\lambda_0>0,q\in (0,1)$ are hyper-parameters and 
	$\xi_k = \frac{1}{n}\sum_{i=1}^m 2a_i^\top x^k\mbox{sign}((a_i^\top x^k)^2-b_i).$ 
 \item[(iii)] \textbf{IPL-FISTA-Low}, \textbf{IPL-FISTA-High}: 
 our IPL+FISTA algorithm with stopping criteria (LACC-FISTA) and (HACC-FISTA)  in Algorithm \ref{Nesterov}, respectively, and we also used the Armijo backtracking line search discussed in Remark \ref{remark-alg-nesterov}.
\end{itemize}
The initial point for all the tested algorithms is generated by the spectral initialization given in Algorithm 3 in \cite{duchi2019solving}.  All the \tcb{code is} run on a server with Intel Xeon E5-2650v4 (2.2GHz). 
Each task is limited to a single core -- no multi-threading is used.

\subsection{Synthetic Data}


We generate synthetic data following the same manner as \cite{duchi2019solving}. Specifically, $a_i$'s are drawn randomly from the normal distribution $\mathcal{N}(0,I_n)$. The entries of $x_\star\in\{-1,1\}^n$ are drawn randomly from discrete Bernoulli distribution. We denote $p_{\fail} = |\mathbb{I}_2|/m$, where $\mathbb{I}_2$ is generated \tcb{by} random sampling without replacement from $\{1,2,\ldots,m\}$. $\xi_i$'s for these samples in \eqref{intro_corrupted_phase} are independently drawn from Cauchy distribution, which means that
$$b_i=\xi_i=\tilde{M}\tan \left(\frac{\pi}{2}U_i\right),U_i\sim U(0,1),\forall\ i\in \mathbb{I}_{2},$$
where $\tilde{M}$ is the sample median of $\{(a_i^\top x_\star)^2\}_{i=1}^m$. 
For a given threshold $\epsilon>0$, we call an algorithm successful if it returns an $x$ such that the relative error 
\begin{equation}\label{relative-error}
    \Delta(x)/\|x_\star\|_2\leq \epsilon.
\end{equation}
For each combination of $n,k = m/n$ and $p_{\fail}$, we randomly generate 50 instances according to the above procedure, and we report the success rate of the algorithm among the 50 instances. 
For IPL-FISTA-Low and IPL-FISTA-High, we set $\rho_l = \rho_h = 0.24.$ For the subgradient method, we set $q = 0.998$ which is one of the suggested choices of $q$ in \cite{davis2018subgradient}. Moreover, \cite{davis2018subgradient} did not specify how to choose $\lambda_0$, and we set $\lambda_0 = 0.1\|x^0\|_2$ as we found that this choice gave good performance. 
Since we found that in most cases, the relative error given by PL is in the level of $[10^{-5},10^{-3}]$, we set $\epsilon=10^{-3}$ in \eqref{relative-error} for PL. 
In our comparison, we first run PL using the default settings of the \tcb{code} provided by the authors of \cite{duchi2019solving}. If the returned $x$ satisfies \eqref{relative-error} with  $\epsilon=10^{-3}$, then we claim that PL is successful, and we terminate IPL and Subgradient method once they found an iterate with a smaller objective value than the one given by PL. If the iterate returned by PL does not satisfy \eqref{relative-error} with $\epsilon=10^{-3}$, then we claim that PL failed, and we then terminate IPL and Subgradient method when $F(x^k) - F(x_\star)\leq 10^{-7}$. The CPU time is only reported for the successful cases for PL. \tcb{The cost of computing the spectral norm $\|A\|_2$ to obtain $L$  is included in the total CPU time of PL and IPL.}

The simulation results corresponding to $p_\fail=0.05$ and $p_\fail=0.15$ are shown in Figures \ref{comparison_gen1} and \ref{comparison_gen2}, where the $x$-axis corresponds to different values of $m$ since $n=500$ is fixed.
From both Figures \ref{comparison_gen1} and \ref{comparison_gen2}, we can see that the four algorithms have similar success rates, but the total CPU time of IPL-FISTA-Low and IPL-FISTA-High \tcb{that includes the cost of computing the spectral norm $\|A\|_2$} is significantly less than that of others.


\begin{figure}[!t]
	\centering
	\includegraphics[scale=0.5]{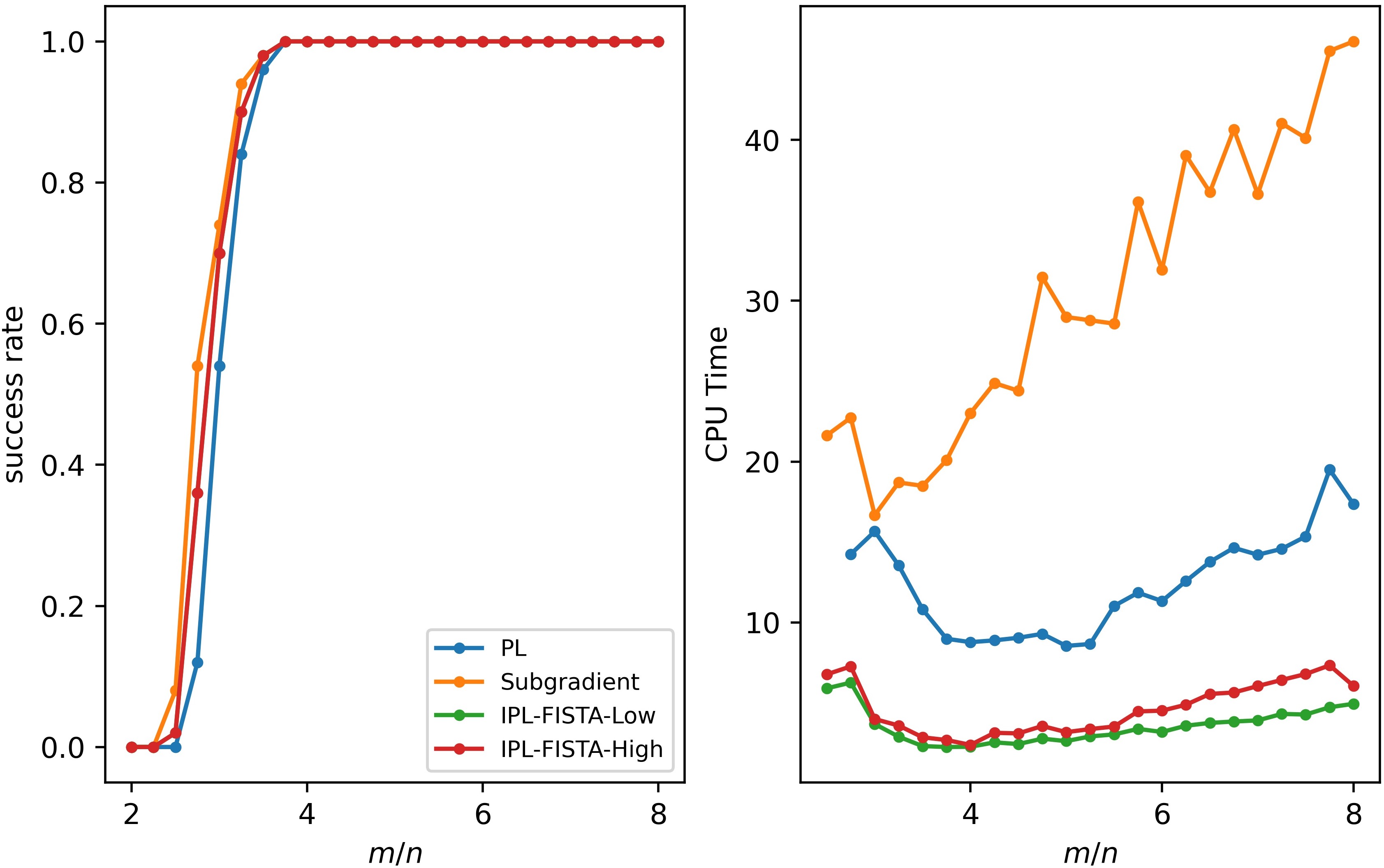}
	\caption{The comparison of success rates and CPU time on synthetic datasets with $p_{\fail} = 0.05$ and $n=500$.}
	\label{comparison_gen1} 
\end{figure}

\begin{figure}[!t]
	\centering
	\includegraphics[scale=0.5]{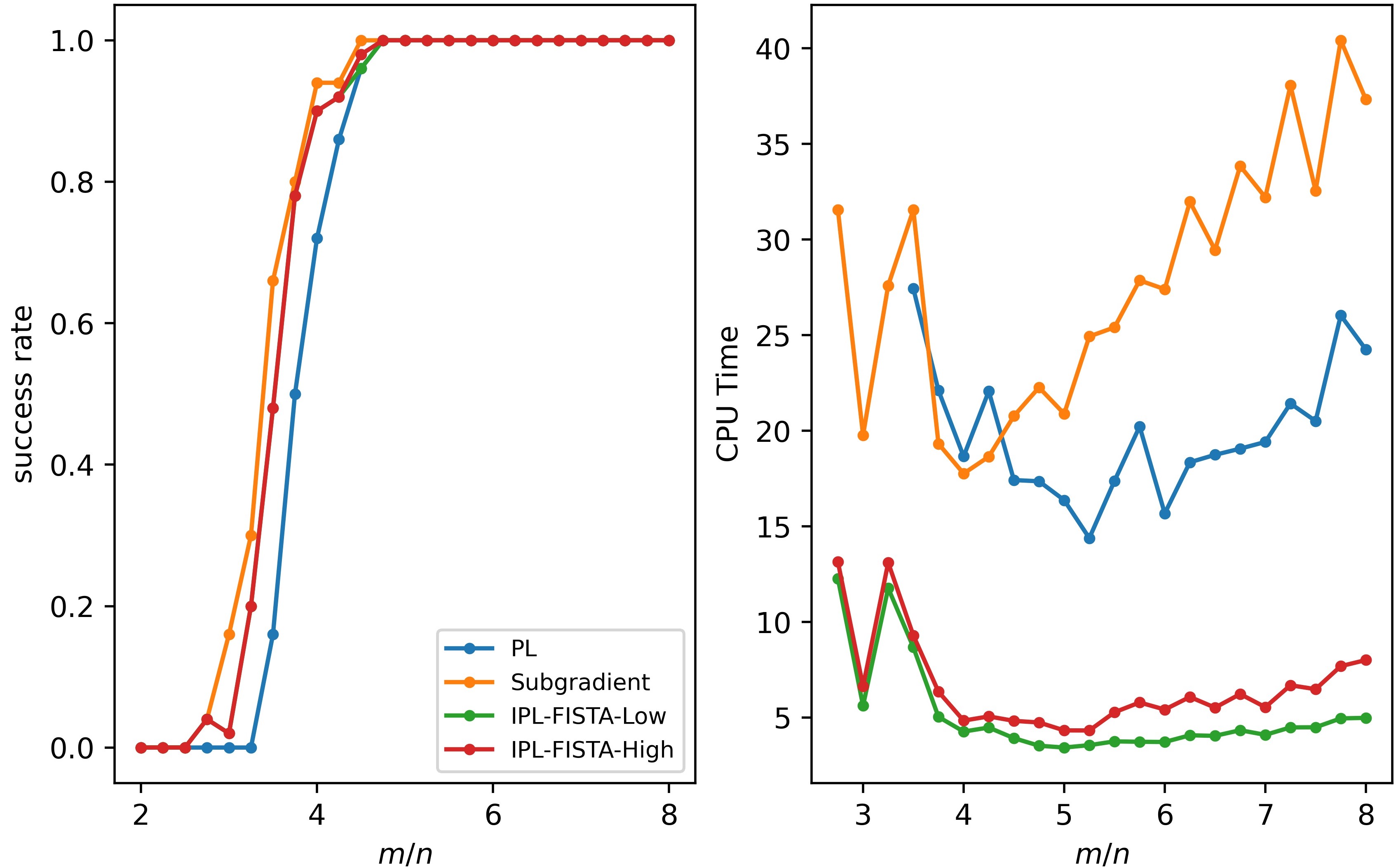}
	\caption{The comparison of success rates and CPU time on synthetic datasets with $p_{\fail} = 0.15$ and $n=500$.}
	\label{comparison_gen2}
\end{figure}

\subsection{\tcb{Image Recovery}}

In this section, we compare the four candidate algorithms on images \tcb{ in a similar manner} as  \cite{duchi2019solving}. In particular, suppose we have an RGB image array $X_\star\in \mathbb{R}^{n_1\times n_2\times 3}$, we construct the signal as $x_\star=[\mbox{vec}(X_{\star});0]\in \mathbb{R}^{n}$, in which 
$n=\min\{2^s\mid s\in N,2^s\geq 3n_1n_2\}$. Let $H_n\in\frac{1}{\sqrt{n}}\{-1,1\}^{n\times n}$ be the Hadamard matrix and $S_j\in \mbox{diag}(\{-1,1\}^n),j=1,2,\ldots,k$ is a random diagonal matrix, and its diagonal elements are independently distributed as discrete uniform distribution. We then let $A=\frac{{\sqrt{m}}}{\sqrt{k}}[H_nS_1;H_nS_2\ldots H_nS_k]$ and we know $L=2$ in this case. The advantage of such a mapping is that it mimics the fast Fourier transform, and calculating $Ay$ is only of time complexity $O(m\log m)$. We first examine the numerical comparisons on a real RNA nanoparticles image\footnote{https://visualsonline.cancer.gov/details.cfm?imageid=11167} as shown in Figure \ref{rna_ori}, and we follow the code of \cite{duchi2019solving} for the experiments on a sub-image with $n=2^{18}$. We also take $p_{\fail}\in\{0.05,0.1,0.15,0.2\}$, $k\in\{3,6\}$ and set the noise in the same way as the synthetic datasets. For each combination of dataset parameters, we run 50 replicates by generating 50 different $A$ and test all the candidate algorithms. We use the same way as the synthetic datasets to define success. For IPL, $\rho_l = \rho_h = 0.24.$ 
For the Subgradient method, $\lambda_0 = 0.1\|x^0\|_2,q = 0.998$. For a replicate, if PL succeeds, the CPU time is the time needed to reach \eqref{relative-error}. 

\begin{figure}[!t]
	\centering
	\includegraphics[width=3.0in]{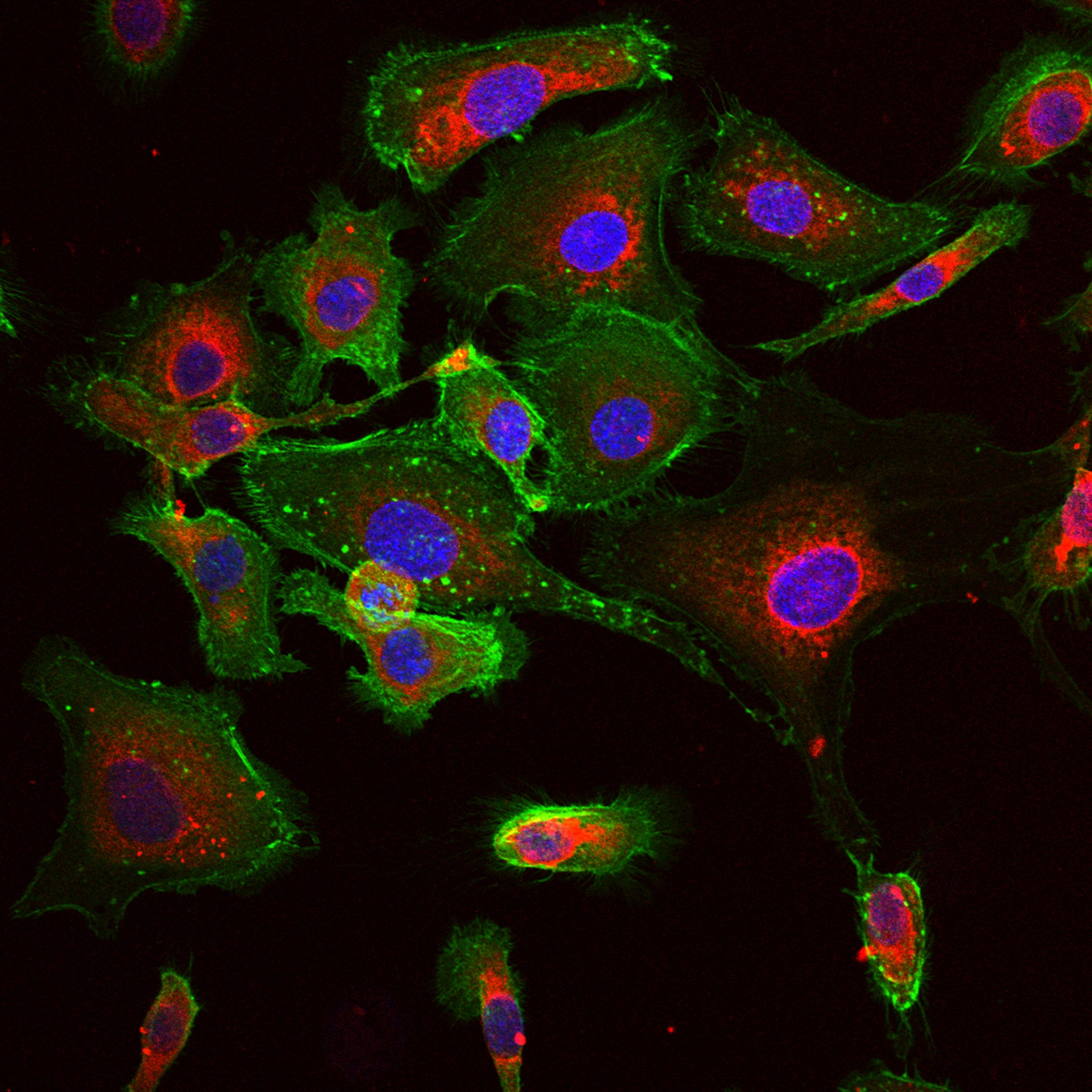}
	\caption{A real RNA nanoparticles image.}
	\label{rna_ori}
\end{figure}


Table \ref{comp_cpu_rna18} reports the median CPU time (in seconds) of the candidate algorithms for $p_{\fail} = 0.1$ and $m/n = 6$ based on two tolerances $\epsilon=10^{-1}$ and $\epsilon=10^{-7}$. We only show the results for this combination of $p_{\fail}$ and $m/n$ because other choices give similar results. It is noted from Table \ref{comp_cpu_rna18} that PL can only reach a relative error that takes value in $[10^{-2},10^{-1}]$, and IPL-FISTA is much more efficient than PL and Subgradient methods.
\begin{table}[]
	\centering
	\caption{The comparison of the median CPU time in seconds for RNA image recovery.}
	\begin{tabular}{|c|c|c|}
	\hline
	 &  $\epsilon = 0.1$& $\epsilon = 10^{-7}$ \\ \hline
	PL             & 3166.93  & NA                 \\ \hline
	Subgradient    & 88.14    & 659.64                 \\ \hline
	IPL-FISTA-Low  & 6.01     & 218.14                     \\ \hline
	IPL-FISTA-High & 48.56    & 175.60                   \\ \hline
	\end{tabular}
	\label{comp_cpu_rna18}
\end{table}


Additional experimental results are reported in Table \ref{multiple_image} for comparing CPU time for the four candidate algorithms. We provide four images with $n$ being at most $2^{22}$. The experiments use the same $m/n$ and $p_{\fail}$, and the CPU time based on ten replications is reported in the form of ``median (Interquartile Range)". Subgradient method, IPL-FISTA-Low and IPL-FISTA-High are terminated with tolerance $\epsilon=10^{-7}$. We can see that PL cannot terminate within 10 hours, and IPL-FISTA still enjoys the best efficiency.


\begin{table}[]
	\centering
	\caption{\tcb{The comparison of the median CPU time in hours with the interquartile range (IQR) in the parentheses for multiple image recovery tasks.}}
\begin{tabular}{|l|l|l|}
\hline
\textbf{}               & \textbf{RNA} ($n = 2^{22}$)       & \textbf{Hubble}\tablefootnote{https://www.nasa.gov/image-feature/goddard/2017/hubble-hones-in-on-a-hypergiants-home} ($n = 2^{22}$)    \\ \hline
\textbf{PL}             & $>10$   & $>10$   \\ \hline
\textbf{Subgradient}    & 4.39 (0.24)                  & 4.72 (0.27)                   \\ \hline
\textbf{IPL-FISTA-Low}  & 1.65 (0.41)                  & 2.07 (0.14)                    \\ \hline
\textbf{IPL-FISTA-High} & 1.30 (0.08)                   & 1.26 (0.06)                    \\ \hline
\textbf{}               & \textbf{James Web}\tablefootnote{https://www.nasa.gov/webbfirstimages} ($n = 2^{21}$) & \textbf{Penn State}\tablefootnote{https://www.britannica.com/topic/Pennsylvania-State-University} ($n = 2^{22}$)\\ \hline
\textbf{PL}             & $>10$                          & $>10$                           \\ \hline
\textbf{Subgradient}    & 1.86 (0.14)                   & 4.84 (0.66)                  \\ \hline
\textbf{IPL-FISTA-Low}  & 0.90 (0.21)                   & 2.23 (0.16)                    \\ \hline
\textbf{IPL-FISTA-High} & 0.56 (0.09)                   & 1.41 (0.21)                    \\ \hline
\end{tabular}
\label{multiple_image}
\end{table}

\section{Proofs}\label{sec:proof}

\subsection{Proof of Theorem \ref{thm:gen_convergence_rate}}

Before proceeding, we first present some lemmas. 

\begin{lemma}[Weak Convexity]\label{thm:gen_weak}
(Discussion of Condition C2 in \cite{duchi2019solving}) The following inequalities hold for any $x,y\in \mathbb{R}^n$, and $L = \frac{2}{m}\|A\|_2^2$. 
	\begin{align}
		\left|F(x)-F(x;y)\right|\leq \frac{L}{2}\|x-y\|_2^2,\label{rel:gen_weak1} \\
		F(x)\leq F_t(x;y), \ \forall \  t\leq \frac{1}{L}.\label{rel:gen_weak2}
	\end{align}
\end{lemma}

\begin{lemma}[see, e.g., \tcb{equation (5.2)} in \cite{drusvyatskiy2019efficiency}]\label{thm:gen_descent0}
	The following inequality holds for any $0<t\leq 1/L$ and $x,y\in \mathbb{R}^n.$
	\begin{equation}\label{thm:gen_descent}
		F(x)-F(y)+\varepsilon(y;x)\geq \frac{1}{2t}\|x-S_t(x)\|_2^2.
	\end{equation}
\end{lemma}


The following lemma studies one iteration of our IPL algorithm (as summarized in Algorithm \ref{alg:adaptive-IPL}).

\begin{lemma}\label{thm:gen_descent_inexact0}
	If $0<t\leq \frac{1}{L}$, we have the following inequalities.
	\begin{enumerate}
		\item[(a)] When the low accuracy condition \eqref{low0} is satisfied, we have
		\begin{equation}\label{thm:gen_descent_low}
			F(x^k)-F(x^{k+1})\geq \frac{1}{2(1+\rho_l)t}\|x^k-S_t(x^k)\|_2^2.
		\end{equation}
		\item[(b)] When the high accuracy condition \eqref{high0} is satisfied, we have
		\begin{equation}\label{thm:gen_descent_high}
			F(x^k)-F(x^{k+1})\geq \frac{1-2\sqrt{\rho_h}}{2t(1-\sqrt{\rho_h})^2}\|x^k-S_t(x^k)\|_2^2.
		\end{equation}
	\end{enumerate}
\end{lemma}

\begin{proof}[Proof of Lemma \ref{thm:gen_descent_inexact0}] We first prove part (a) of Lemma \ref{thm:gen_descent_inexact0} and then prove part (b) of Lemma \ref{thm:gen_descent_inexact0}. 

(a). Letting $x=x^k$ and $y=x^{k+1}$ in \eqref{thm:gen_descent}, we have 
		\begin{align*}
			F(x^k)&-F(x^{k+1})\geq \frac{1}{2t}\|x^k-S_t(x^k)\|_2^2 - \varepsilon(x^{k+1};x^k)\\
			&\geq \frac{1}{2t}\|x^k-S_t(x^k)\|_2^2-\rho_l(F(x^k)-F(x^{k+1})),
		\end{align*}
  where the second inequality follows from \eqref{low0} and \eqref{rel:gen_weak2}. 
		This proves \eqref{thm:gen_descent_low} in Lemma \ref{thm:gen_descent_inexact0}(a).
  
  (b). When \eqref{high0} holds, since $F_t(\cdot;x^{k})$ is $\frac{1}{t}$-strongly convex, we have 
		\begin{equation}\label{proof-lemma3-eq-1}
			\frac{\rho_h}{2t}\|x^k-x^{k+1}\|_2^2\geq \varepsilon(x^{k+1};x^k)\geq \frac{1}{2t}\|x^{k+1}-S_t(x^k)\|_2^2.
		\end{equation}
		Let $u=x^k-x^{k+1}$, $v=x^{k+1}-S_t(x^k)$. From \eqref{proof-lemma3-eq-1} we have 
    \begin{equation}\label{proof-lemma3-eq-2}
        \rho_h\|u\|_2^2\geq \|v\|_2^2,
    \end{equation} 
    and
		\begin{align}
			&\|u+v\|_2^2-(1-\sqrt{\rho_h})^2\|u\|_2^2 \nonumber \\ 
   = & (2\sqrt{\rho_h}-\rho_h)\|u\|_2^2+ 
			2u^\top v+\|v\|_2^2 \nonumber \\ 
   \geq &  \sqrt{\rho_h}(2-\sqrt{\rho_h})\|u\|_2^2-2\|u\|_2\|v\|_2+\|v\|_2^2 \nonumber \\
			= & \left(\sqrt{\rho_h}\|u\|_2-\|v\|_2\right)\left((2-\sqrt{\rho_h})\|u\|_2-\|v\|_2\right) \nonumber  \\ 
   \geq & 0,\label{proof-lemma3-eq-3}
		\end{align}
  where the first inequality is from the Cauchy-\tcr{Schwarz} inequality, and the second inequality follows from \eqref{proof-lemma3-eq-2} and the fact that $\rho_h\in (0,1/4)$. Therefore, from \eqref{proof-lemma3-eq-1} and \eqref{proof-lemma3-eq-3} we have 
		\begin{align*}
			\varepsilon(x^{k+1};x^k)&\leq \frac{\rho_h}{2t}\|x^k-x^{k+1}\|_2^2\\
			&\leq \frac{\rho_h}{2t(1-\sqrt{\rho_h})^2}\|x^k-S_t(x^k)\|_2^2,
		\end{align*}
		which, together with \eqref{thm:gen_descent}, yields
		\begin{align*}
		&	F(x^k)-F(x^{k+1}) \\
  \geq & \frac{1}{2t}\|x^k-S_t(x^k)\|_2^2 - \varepsilon(x^{k+1};x^k) \\
			\geq & \frac{1}{2t}\|x^k-S_t(x^k)\|_2^2-\frac{\rho_h}{2t(1-\sqrt{\rho_h})^2}\|x^k-S_t(x^k)\|_2^2 \\
   = & \frac{1-2\sqrt{\rho_h}}{2t(1-\sqrt{\rho_h})^2}\|x^k-S_t(x^k)\|_2^2.
		\end{align*}
		This proves \eqref{thm:gen_descent_high} in Lemma \ref{thm:gen_descent_inexact0}(b). 
  
  Therefore, the proof of Lemma \ref{thm:gen_descent_inexact0} is complete.  
\end{proof}
Now we are ready to prove Theorem \ref{thm:gen_convergence_rate}. 

\begin{proof}[Proof of Theorem \ref{thm:gen_convergence_rate}]
\tcb{We will prove (a) and (b) together.} Both \eqref{thm:gen_descent_low} and \eqref{thm:gen_descent_high} indicate that $\{x^k\}$ generated by Algorithm \ref{alg:adaptive-IPL} satisfies
	\begin{equation*}
		F(x^k)-F(x^{k+1})\geq \beta t\|\mathcal{G}_t(x^k)\|_2^2,
	\end{equation*}
	in which $\beta>0$ is a constant \tcb{that $\beta = 1/(2(1+\rho_l))$ for LACC and $\beta = (1-2\sqrt{\rho_h})/(2(1-\sqrt{\rho_h})^2)$ for HACC}. Letting $F^\star = \inf_{x\in \mathbb{R}^{n}} F(x)$, we have
	\begin{align*}
		F(x^0)-F^\star& \geq \sum_{k=0}^{K_0-1}\beta t\|\mathcal{G}_t(x^k)\|_2^2\\
		&\geq \beta tK_0\min_{0\leq k\leq K_0-1}\|\mathcal{G}_t(x^k)\|_2^2.
	\end{align*}
	Then, our IPL (Algorithm \ref{alg:adaptive-IPL}) finds an $\epsilon$-stationary point to RPR \eqref{RPR} in  
\tcb{	\begin{equation*}
		\left\lfloor\frac{F(x^0)-F^\star}{\beta t\epsilon^2}\right\rfloor
	\end{equation*}}
 iterations. Therefore, the proof of \tcb{both (a) and (b) in} Theorem \ref{thm:gen_convergence_rate} is complete. 
\end{proof}

\subsection{Proof of Theorem \ref{thm:convergence_rate_sharpness}}


To prove  Theorem \ref{thm:convergence_rate_sharpness}, we need the following lemmas.

\begin{lemma}\label{descent_general}
	(Part of the Proof for Theorem 1 in \cite{duchi2019solving}) Let $t = 1/L$. For any $x^k,x^{k+1}\in \mathbb{R}^n$, we have
	\begin{align}
		&F(x^{k+1}) - F(x_\star)+\frac{L}{4}\|x^{k+1} - x_\star\|_2^2 \nonumber \\
		&\leq L\|x^k - x_\star\|_2^2+2\varepsilon(x^{k+1},x^k),\label{ineq:descent_general}
	\end{align}
	which also holds if we replace $x_\star$ by $-x_\star$.
\end{lemma}

\begin{proof}[Proof of Lemma \ref{descent_general}]
\tcb{We have}
	\begin{align}
	& \frac{L}{2}\|x_\star-S_t(x^k)\|_2^2\nonumber \\
 \geq & \frac{L}{4}\|x_\star-x^{k+1}\|_2^2 - \frac{L}{2}\|S_t(x^{k})-x^{k+1}\|_2^2 \nonumber\\
	\geq & \frac{L}{4}\|x_\star-x^{k+1}\|_2^2-\varepsilon(x^{k+1};x^k).\label{proof-lemma4-eq-1}
	\end{align}
 \tcb{where the first inequality follows from the Cauchy-Schwarz inequality} and the second one is from the convexity of $\|\cdot\|_2^2$. We then have
	\begin{align*}
		&F(x^{k+1}) \\ 
  \leq & F(x^{k+1};x^k)+\frac{L}{2}\|x^k-x^{k+1}\|_2^2\\
  = & F_t(S_t(x^k);x^k)+\varepsilon(x^{k+1};x^k)\\
		\leq & F(x_\star;x^k)+\frac{L}{2}\|x^k-x_\star\|_2^2-\frac{L}{2}\|x_\star-S_t(x^k)\|_2^2 \\
   & +\varepsilon(x^{k+1};x^k)\\
		\leq & F(x_\star)+L\|x^k-x_\star\|_2^2-\frac{L}{2}\|x_\star-S_t(x^k)\|_2^2+\varepsilon(x^{k+1};x^k),
	\end{align*}
 where the first and the last inequalities are from \eqref{rel:gen_weak1}, and the second inequality is from the strong convexity of $F_t(\cdot;x^k)$. Combining this inequality with \eqref{proof-lemma4-eq-1} yields \eqref{ineq:descent_general}. \tcb{It is easy to find that all the proofs still hold if we replace $x_\star$ with $-x_\star$.} Therefore, the proof of Lemma \ref{descent_general} is complete.
\end{proof}


\begin{lemma}[One-step progress]\label{descent_low_high}
	 Let $t = 1/L$ and suppose that Assumption \ref{ass:sharpness} holds.
	\begin{enumerate}
		\item[(a)] (Low accuracy condition) When \eqref{low0} is satisfied for some $\rho_l\geq 0$, we have
		\begin{align}\label{lemma5-a-eq-1}
			&(1+2\rho_l)\left(F(x^{k+1})-F(x_\star)\right)\\ 
   \leq &2\rho_l (F(x^{k}) - F(x_\star))+L(\Delta(x^k))^2.\nonumber
		\end{align}
		If we also have $F(x^k) - F(x_\star)\leq \lambda_s^2/(2L)$, then we have
		\begin{align}\label{lemma5-a-eq-2}
		    & (1+2\rho_l)\left(F(x^{k+1})-F(x_\star)\right) \\
      \leq & \left(\frac{1}{2}+2\rho_l\right) (F(x^{k}) - F(x_\star)).\nonumber
		\end{align}
		\item [(b)] (High accuracy condition) When \eqref{high0} is satisfied, we have
		\begin{equation}\label{lemma5-b-eq}
		    \lambda_s\Delta(x^{k+1})\leq \frac{L(1-3\rho_h)}{(1-4\rho_h)}(\Delta(x^k))^2. 
		\end{equation}
	\end{enumerate}
\end{lemma}

\begin{proof}[Proof of Lemma \ref{descent_low_high}] We first prove part (a) of Lemma \ref{descent_low_high} and then prove part (b) of Lemma \ref{descent_low_high}. 

(a). When the low accuracy condition \eqref{low0} holds, from \eqref{rel:gen_weak2} we have
		\begin{align*}
			\varepsilon(x^{k+1};x^k)&\leq  \rho_l(F_t(x^k;x^k)-F_t(x^{k+1};x^k))\\
			&\leq  \rho_l(F(x^k)-F(x^{k+1})),
		\end{align*}
  which, combining with  \eqref{ineq:descent_general}, yields
		\begin{align*}
			&(1+2\rho_l)\left(F(x^{k+1})-F(x_{\star})\right)+\frac{L}{4}\|x^{k+1}-x_{\star}\|_2^2\\
			\leq & 2\rho_l\left(F(x^k)-F(x_{\star})\right)+ L\|x^k-x_{\star}\|_2^2.
		\end{align*}
		Discarding the term $\frac{L}{4}\|x^{k+1}-x_{\star}\|_2^2$, we get 
		\begin{align*}
			&(1+2\rho_l)\left(F(x^{k+1})-F(x_{\star})\right)\\
			\leq & 2\rho_l\left(F(x^k)-F(x_{\star})\right)+ L\|x^k-x_{\star}\|_2^2.
		\end{align*}
		Since \eqref{ineq:descent_general} also holds when $x_\star$ is replaced by $-x_\star$, we also have
		\begin{align*}
			&(1+2\rho_l)\left(F(x^{k+1})-F(x_{\star})\right)\\
			\leq & 2\rho_l\left(F(x^k)-F(x_{\star})\right)+ L\|x^k+x_{\star}\|_2^2.
		\end{align*}
		This proves \eqref{lemma5-a-eq-1}. \eqref{lemma5-a-eq-2} holds because of Assumption \ref{ass:sharpness}. Hence, it proves part (a) of Lemma \ref{descent_low_high}. 
  
(b). When the high accuracy condition \eqref{high0} holds, we have
		\begin{align*}
			&2\varepsilon(x^{k+1};x^k) \\ \leq & \rho_h L\|x^k-x^{k+1}\|_2^2\\
		\leq & \rho_h L\left(\frac{\|x^{k+1}-x_\star\|_2^2}{4\rho_h}+\frac{\|x^k-x_\star\|_2^2}{1-4\rho_h}\right),
		\end{align*}
		where the second inequality is from the Cauchy-Schwarz inequality. Combining with \eqref{ineq:descent_general}, we have
		\begin{equation*}
			F(x^{k+1})-F(x_\star)\leq \frac{(1-3\rho_h)L}{1-4\rho_h}\|x^k-x_\star\|_2^2.
		\end{equation*}
		Similarly, replacing $x_\star$ by $-x_\star$, we have
  \begin{equation*}
			F(x^{k+1})-F(x_\star)\leq \frac{(1-3\rho_h)L}{1-4\rho_h}\|x^k+x_\star\|_2^2.
		\end{equation*}
  Combining the above two inequalities with Assumption \ref{ass:sharpness} yields \eqref{lemma5-b-eq}, which proves part (b) of Lemma \ref{descent_low_high}.

  Therefore, the proof of Lemma \ref{descent_low_high} is complete.
\end{proof}
Now we are ready to give the proof of Theorem \ref{thm:convergence_rate_sharpness}.
\begin{proof} [Proof of Theorem \ref{thm:convergence_rate_sharpness}]
(a). We will prove that for any $k\in \mathbb{N}$,
		\begin{equation}\label{proof-thm2-eq-1}F(x^k) - F(x_\star)\leq \left(F(x^0)-F(x_\star)\right)\left(\frac{1+4\rho_l}{2+4\rho_l}\right)^k
  \end{equation}
		by induction, which immediately leads to the conclusion of (a) by Assumption \ref{ass:sharpness}. First, \eqref{proof-thm2-eq-1} clearly holds for $k = 0$. Now we assume that it holds for $k$. For $k+1$, since \eqref{cond:init_low} holds, we have
  \begin{align*}
  & F(x^k) - F(x_\star) \\
  \leq & \left(F(x^0)-F(x_\star)\right)\left(\frac{1+4\rho_l}{2+4\rho_l}\right)^k \\
  \leq & F(x^0)-F(x_\star) \\
  \leq & \lambda_s^2/(2L).
  \end{align*}
  Using \eqref{lemma5-a-eq-2} directly proves that \eqref{proof-thm2-eq-1} holds when $k$ is replaced by $k+1$. This proves part (a) of Theorem \ref{thm:convergence_rate_sharpness}.

(b). Note that \eqref{lemma5-b-eq} is equivalent to
\begin{equation}\label{lemma5-b-eq-equiv}
    \frac{L\Delta(x^{k+1})(1-3\rho_h)}{\lambda_s(1-4\rho_h)}\leq \left(\frac{L\Delta(x^{k})(1-3\rho_h)}{\lambda_s(1-4\rho_h)}\right)^2.
\end{equation}
Since \eqref{cond:init_high} holds, from \eqref{lemma5-b-eq-equiv} we know that
\begin{equation*}
    \frac{L\Delta(x^{1})(1-3\rho_h)}{\lambda_s(1-4\rho_h)}\leq \left(\frac{1}{2}\right)^2.
\end{equation*}
Plug this back into \eqref{lemma5-b-eq-equiv} and work recursively, we get
\begin{equation*}
    \frac{L\Delta(x^{k})(1-3\rho_h)}{\lambda_s(1-4\rho_h)}\leq \left(\frac{1}{2}\right)^{2^k}.
\end{equation*}
Then, $\Delta(x^k)\rightarrow 0$, which together with \eqref{lemma5-b-eq-equiv} proves the quadratic convergence in part (b) of Theorem \ref{thm:convergence_rate_sharpness}.

Therefore, the proof of Theorem \ref{thm:convergence_rate_sharpness} is complete.
\end{proof}

\subsection{Proof of Theorem \ref{overall_iplnesterov}}
Throughout this subsection, we assume that the assumptions in Theorem \ref{overall_iplnesterov} hold. That is, we assume $t = 1/L$, $t_{kj} = \left(t\|B_k\|_2^2\right)^{-1}$ in Algorithm \ref{Nesterov}, and Assumption \ref{ass:sharpness} holds.
To prove Theorem \ref{overall_iplnesterov}, we first present some lemmas. 
\begin{lemma}[Local Lipschitz Constant of $F(x)$]\label{Lip_F}
	 It holds that
	$$\sup_{x,y\in \mathbb{R}^n,\Delta(x)\leq r,\Delta(y)\leq r,x\neq y} \frac{|F(x) - F(y)|}{\|x-y\|_2}\leq L(\|x_\star\|_2 + r).$$
\end{lemma}
\begin{proof}[Proof of Lemma \ref{Lip_F}]
	Denote $u = (x-y)/\|x-y\|_2$ and $v = (x+y)/\|x+y\|_2$ when $x+y\neq 0$. $v$ is set as 0 when $x+y = 0$.
	\begin{align*}
		&\left|F(x) - F(y)\right|/\|x-y\|_2 \\
		=&\frac{1}{m\|x-y\|_2}\left|\sum_{i=1}^m\left|(a_i^\top x)^2-b_i\right|-\sum_{i=1}^m\left|(a_i^\top y)^2-b_i\right|\right|\\
		\leq & \frac{1}{m\|x-y\|_2}\sum_{i=1}^m |(a_i^\top x)^2-(a_i^\top y)^2|\\
	 = & \frac{1}{m\|x-y\|_2}\sum_{i=1}^m |(a_i^\top(x-y))(a_i^\top(x+y))|\\
	= & \|x+y\|_2\frac{1}{m}\sum_{i=1}^m |a_i^\top u|\cdot|a_i^\top v|.
	\end{align*}
	Recalling that $L = \frac{2}{m}\|A\|_2^2 = \frac{2}{m}\|\sum_{i=1}^m a_ia_i^\top\|_2$ as defined in Algorithm \ref{alg:adaptive-IPL} and noticing the fact that when $\Delta(x)\leq r$ and $\Delta(y)\leq r$, we have $\|x+y\|_2\leq {2}(\|x_\star\|_2+r)$, and hence we can claim that
	\begin{align*}
		&\|x+y\|_2\frac{1}{m}\sum_{i=1}^m |a_i^\top u|\times|a_i^\top v|\\
		\leq & \|x+y\|_2\left(u^\top\left(\frac{1}{2m}\sum_{i=1}^m a_ia_i^\top\right)u + v^\top\left(\frac{1}{2m}\sum_{i=1}^m a_ia_i^\top\right)v\right)\\
		\leq & L(\|x_\star\|_2 + r),
	\end{align*}
 which proves the desired result. Therefore, the proof of Lemma \ref{Lip_F} is complete.
\end{proof}

\begin{lemma}\label{lip_F2}
	If Assumption \ref{ass:sharpness} holds, then for any $r\geq 0$, we have
	\begin{align*}
		&\{x\in\mathbb{R}^n:\Delta(x)\leq E(r)\} \\ \subseteq & \{x\in\mathbb{R}^n: F(x) - F(x_\star)\leq r\}\\
		\subseteq & \{x\in\mathbb{R}^n: \Delta(x)\leq r/\lambda_s\},
	\end{align*}
	in which 
 \begin{align}\label{define-E(r)}
 E(r) = \left(\sqrt{L^2\|x_\star\|_2^2+4rL} - L\|x_\star\|_2\right)/(2L).
 \end{align}
 This relationship also indicates that $E(r)\leq r/\lambda_s.$
\end{lemma}
\begin{proof}[Proof of Lemma \ref{lip_F2}]
For $x\in \mathbb{R}^n$, if $\Delta(x)\leq E(r)$ for some $r\geq 0$, without loss of generality, we assume that $\Delta(x) = \|x-x_\star\|_2$. From Lemma \ref{Lip_F}, we have 
\[
F(x) - F(x_\star)\leq LE(r)(\|x_\star\|_2+E(r)) = r,
\]
which proves the first inclusion. The second inclusion follows immediately from Assumption \ref{ass:sharpness}. Therefore, the proof of Lemma \ref{lip_F2} is complete.
\end{proof}

\begin{lemma}[Bound of $\|B_k\|_2$]\label{bound_b2}
	For any $r\geq 0$, if $\sup_{k\in\mathbb{N}}\Delta(x^k)\leq r$, then
	$$\sup_{k\in\mathbb{N}} \|B_k\|_2\leq B(r):=\frac{2}{m}\|A\|_2(\|x_\star\|_2+r)\max_{i = 1,2,\ldots,m}\|a_i\|_2.$$
\end{lemma}
\begin{proof}[Proof of Lemma \ref{bound_b2}]
	Since $B_k = \frac{2}{m}\mbox{diag}(Ax^k)A$, we have $\|B_k\|_2\leq \frac{2}{m}\|Ax^k\|_\infty \|A\|_2$
	and $\|x^k\|_2\leq \|x_\star\|_2+r.$ The desired result follows by using $\|Ax^k\|_\infty\leq (\|x_\star\|_2+r)\max_{i = 1,2,\ldots,m}\|a_i\|_2$. Therefore, the proof of Lemma \ref{bound_b2} is complete.
\end{proof}

Next, we provide Lemmas \ref{conv_rate_nesterov} and \ref{lemma:primal_recovery} to show the convergence rate for solving the subproblem with Algorithm \ref{Nesterov}. These results can be used for both conditions for (a) and (b).
\begin{lemma}[see \cite{nesterov1988}]\label{conv_rate_nesterov}
In the $j$-th iteration of Algorithm \ref{Nesterov}, we have
	\begin{align*}
		\max_{\|\lambda\|_\infty\leq 1} D_k(\lambda) - D_k(\lambda_a^j)&\leq \frac{tC\|B_k\|_2^2}{(j+1)^2}\|\lambda^0 - \lambda_{k+1}^\star\|_2^2\\
		&\leq \frac{tCm\|B_k\|_2^2}{(j+1)^2},\forall j\in\mathbb{N},
	\end{align*}
	where $\lambda_{k+1}^\star\in \argmax_{\|\lambda\|_\infty\leq 1}\ D_k(\lambda)$ and $C>0$ is universal constant.
\end{lemma}


\begin{lemma}[Theorem 4 in \cite{dunner2016primal}]\label{lemma:primal_recovery}
For Algorithm \ref{Nesterov}, there exist universal positive constants $C',C''$ such that, when $j\geq C',j\in \mathbb{N}$, it holds
	\begin{align}\label{lemma10-conclud} H_k(z_k(\lambda_a^j)) - D_k(\lambda_a^j)\leq \frac{C''tm\|B_k\|_2^2}{j+1}.
 \end{align}
\end{lemma}

\begin{proof}[Proof of Lemma \ref{lemma:primal_recovery}]
Theorem 4 in \cite{dunner2016primal} indicates that if
\begin{align*}
(j+1)^2 & \geq \max\left\{\frac{2Ctm\|B_k\|_2^2/t}{m\|B_k\|_2^2},\frac{2Ctm\|B_k\|_2^2\|B_k\|_2^2m}{(1/t)\epsilon^2}\right\} \\ & = \max\left\{2C,\frac{2Ct^2m^2\|B_k\|_2^4}{\epsilon^2}\right\}
\end{align*}
then we have
$$H_k(z_k(\lambda_a^j)) - D_k(\lambda_a^j) \leq \epsilon.$$ Here $C$ is the constant used in Lemma \ref{conv_rate_nesterov}.
Specifically, by choosing $\epsilon = \frac{tm\|B_k\|_2^2\sqrt{2C}}{j+1}$, we know that if $j\geq C' := \sqrt{2C}-1$ holds, then \eqref{lemma10-conclud} holds, where $C'' = \sqrt{2C}$. Therefore, the proof of Lemma \ref{lemma:primal_recovery} is complete.
\end{proof}

We now define some constants. 
		$$E_1 = \min\{E\left(\lambda_s^2/(2L)\right),E\left(\lambda_sM_1B^2(\lambda_s/(2L))/C'\right)\},$$
		$$E_2 = \frac{(2+4\rho_l)M_1B^2(\lambda_s/(2L))L(\|x_\star\|_2+\lambda_s/(2L))}{\lambda_s},$$
\[
E_3 = \min\{E_0/2,B(E_0/2)\sqrt{M_2/C'},E(\lambda_s^2/(2L))\},
\]
\[E_4= 4M_2B^2(E_0/2)/(3E_0),
\]
and $\{\epsilon_i\}_{i=1}^\infty = \{\Delta(x^i)\}_{i=1}^\infty,$ where $M = \frac{\lambda_s}{2L}\left(\frac{\lambda_s}{2L} + \|x_\star\|_2\right)^{-1}$, $M_1 = {2}C''tm(\rho_l+1)/(\lambda_s\rho_l)$, $M_2 = 4C''t^2m(\rho_h+1)/(\rho_hM^2)$, $E_0 = \frac{\lambda_s(1-4\rho_h)}{(1-3\rho_h)L}$ and $C',C''$ are universal positive constants mentioned in Lemma \ref{lemma:primal_recovery}.

We now formally state the sufficiency of \eqref{low2} and \eqref{high2} for \eqref{low0} and \eqref{high0} in Lemma \ref{sufficiency_nesterov} so that we can use the linear and quadratic convergence rate for main iterations. 

\begin{lemma}\label{sufficiency_nesterov}
	For Algorithm \ref{Nesterov}, \eqref{low2} indicates \eqref{low0} and \eqref{high2} indicates \eqref{high0}.
\end{lemma}

\begin{proof}[Proof of Lemma \ref{sufficiency_nesterov}]
Note that $z_k(\lambda_a^{j+1}) = -tB_k^\top\lambda_a^{j+1}$, and $x^{k+1} = x^k + z_k(\lambda_a^{j+1})$. 
	The conclusion immediately holds because $H_k(z_k(\lambda_a^{j+1})) = F_t(x^{k+1};x^k)$ and
	$$H_k(z_k(\lambda_a^{j+1})) - D_k(\lambda_a^{j+1}))\geq H_k(z_k(\lambda_a^{j+1})) - \min_{z\in\mathbb{R}^n}H_k(z)$$
	which is from strong duality. Therefore, the proof of Lemma \ref{sufficiency_nesterov} is complete.
\end{proof}

Based on Lemma \ref{sufficiency_nesterov}, we prove some properties induced by conditions of Theorem \ref{overall_iplnesterov}. In particular, Lemma \ref{sup_of_seq_low} gives the ones induced by part (a) of Theorem \ref{overall_iplnesterov}, and Lemma \ref{sup_of_seq_high} gives the ones induced by part (b) of Theorem \ref{overall_iplnesterov}. 

\begin{lemma}\label{sup_of_seq_low}
	Assume \tcb{Assumption \ref{ass:sharpness}} and \eqref{low0} \tcr{hold} for any $k\in\mathbb{N}$ with some $\rho_l\geq 0$. If $\Delta(x^0)\leq E_1$, then for any $k\in\mathbb{N}$, we have
\begin{equation}\label{lemma-14-equation-1}
     F(x^k) - F(x_\star)\leq \lambda_s^2/(2L)
 \end{equation}
 and 
 \begin{equation}\label{lemma-14-equation-2}
     \Delta(x^k)\leq \min\{\lambda_s/(2L),M_1B^2(\lambda_s/(2L))/C'\}.
 \end{equation}
	
	\begin{proof}[Proof of Lemma 
 \ref{sup_of_seq_low}]
Note that $E(\cdot)$ defined in \eqref{define-E(r)} is monotonically increasing. Since $\Delta(x^0)\leq E_1$, from Lemma \ref{lip_F2} we have
    \[F(x^0)-F(x_\star)\leq \lambda_s\min\{\lambda_s/(2L),M_1B^2(\lambda_s/(2L))/C'\}.\]
    Therefore, we can apply \eqref{lemma5-a-eq-2} and it implies that 
    \begin{align}\label{lemma14-proof-1}
    F(x^k) - F(x_\star)&\leq F(x^0) - F(x_\star)\\ \nonumber
    &\leq \lambda_s\min\{\lambda_s/(2L),M_1B^2(\lambda_s/(2L))/C'\},
    \end{align}
    which proves \eqref{lemma-14-equation-1}. From \eqref{ineq:sharpness}, we have
		$$\lambda_s\Delta(x^k)\leq F(x^k) - F(x_\star),$$
  which leads to \eqref{lemma-14-equation-2}.  Therefore, the proof of Lemma \ref{sup_of_seq_low} is complete.
	\end{proof}
	
\end{lemma}

\begin{lemma}\label{sup_of_seq_high}
	Assume \tcb{Assumption \ref{ass:sharpness}} and \eqref{high0} \tcr{hold} for any $k\in\mathbb{N}$ with some \tcb{$0\leq \rho_h<1/4$}. If $\Delta(x^0)\leq E_3$, then for any $k\in\mathbb{N}$, \eqref{lemma-14-equation-1} holds, and the following two inequalities hold
  \begin{align}\label{lemma-14-equation-5}
      \Delta(x^{k+1}) & \leq \frac{1}{2}\Delta(x^k) \\ 
  \label{lemma-14-equation-4}
     \Delta(x^k) & \leq E_3.
 \end{align}
	
\begin{proof}[Proof of Lemma 
 \ref{sup_of_seq_high}]
Based on \eqref{lemma5-b-eq}, we have
  \begin{equation}\label{proof-lemma14-eq-1}
      E_0\Delta(x^{k+1})\leq \Delta^2(x^k),\forall k\geq 0.
  \end{equation} 
We prove \eqref{lemma-14-equation-5} by induction. When $k=0$, noticing that $\Delta(x^0)\leq E_0/2$, by \eqref{proof-lemma14-eq-1}, we have $E_0\Delta(x^1)\leq \Delta(x^0)(E_0/2)$, and therefore \eqref{lemma-14-equation-5} holds for $k=0$. We now assume that \eqref{lemma-14-equation-5} holds for $k<k_0$ for $k_0\in\mathbb{N}$. We then have $\Delta(x^{k_0})\leq \Delta(x^0)\leq E_0/2$ and using \eqref{proof-lemma14-eq-1}, we have
  $E_0\Delta(x^{k_0+1})\leq \Delta(x^{k_0})(E_0/2)$, which implies that \eqref{lemma-14-equation-5} holds for $k=k_0$. This completes the proof for \eqref{lemma-14-equation-5}, which immediately indicates \eqref{lemma-14-equation-4} for any $k\geq 0$. This indicates that $\Delta(x^k)\leq \Delta(x^0)\leq E_3\leq E(\lambda_s^2/(2L))$. By Lemma \ref{lip_F2}, $F(x^k) - F(x_\star)\leq \lambda_s^2/(2L)$, which indicates that \eqref{lemma-14-equation-1} holds for any $k\geq 0$.  Therefore, the proof of Lemma \ref{sup_of_seq_high} is complete.
	\end{proof}	
\end{lemma}

We can notice that a common property under conditions for (a) or (b) is that \eqref{lemma-14-equation-1} holds for any $k\geq 0$, based on which, we give Lemma \ref{thm:sharp_diff2} to show another common property.
\begin{lemma}\label{thm:sharp_diff2}
If \eqref{lemma-14-equation-1} and \tcb{Assumption \ref{ass:sharpness}} hold, then we have 
	\begin{align}\label{lemma8-conclusion-1}
	   & \frac{\lambda_s}{2L}\left(\frac{\lambda_s}{2L} + \|x_\star\|_2\right)^{-1}\Delta(x^k)\leq \|x^k - S_t(x^k)\|_2,\\
	& \label{lemma8-conclusion-2}{ \frac{1}{2}}\lambda_s\Delta(x^k)\leq F_{t}(x^k;x^k) - F_{t}\left(S_t(x^k);x^k\right).
 \end{align}
\end{lemma}

\begin{proof}[Proof of Lemma \ref{thm:sharp_diff2}]
 
For fixed $k\geq 0$, we define \tcb{$\tilde{x}^0 = x^k$ and $\tilde{x}^1 = S_t(x^k)$.}
Therefore $\tilde{x}^0$ and $\tilde{x}^{1}$ satisfy \eqref{low0} (with $k$ replaced by $0$ and $x$ replaced by $\tilde{x}$) with $\rho_l = 0$.
Hence we have $\lambda_s\Delta(\tilde{x}^0)\leq F(\tilde{x}^0)-F(x_\star)\leq \lambda_s^2/(2L)$ and 
	\begin{align}
		&F(\tilde{x}^0) - F(\tilde{x}^1) \nonumber \\
  = & F(\tilde{x}^0) - F(x_\star) - \left(F(\tilde{x}^1) - F(x_\star)\right)\nonumber \\
\geq & F(\tilde{x}^0) - F(x_\star) - \frac{1}{2}\left(F(\tilde{x}^0) - F(x_\star)\right)\nonumber \\
= & \frac{1}{2}(F(\tilde{x}^0)-F(x_\star))\label{proof-lemma8-eq-1}
	\end{align}
	 where the inequality is from \eqref{lemma5-a-eq-2} with $\rho_l=0$. Moreover, from Assumption \ref{ass:sharpness} we have 
\tcb{\begin{align*}
\lambda_s\Delta(\tilde{x}^{1})&\leq F(\tilde{x}^{1}) - F(x_\star)\leq (1/2)\left(F(\tilde{x}^0) - F(x_\star)\right)\\
		&\leq \lambda_s^2/(2L),
\end{align*}}
where the second inequality can be obtained by applying \eqref{lemma5-a-eq-2}.
\tcb{Thus, we have $\max\{\Delta(\tilde{x}^0),\Delta(\tilde{x}^1)\}\leq \lambda_s/(2L)$}.	So, based on Lemma \ref{Lip_F} and Assumption \ref{ass:sharpness},
	\begin{align} \lambda_s\Delta(\tilde{x}^0)&\leq F(\tilde{x}^0) - F(x_\star)\leq 2\left(F(\tilde{x}^0)-F(\tilde{x}^{1})\right) \nonumber\\
		&\leq 2L\left(\frac{\lambda_s}{2L}+\|x_\star\|_2\right)\|\tilde{x}^0 - \tilde{x}^1\|_2,\label{proof-lemma8-eq-2}
	\end{align}
 where the second inequality is from \eqref{proof-lemma8-eq-1}, and the last inequality is from Lemma \ref{Lip_F}.
 \eqref{proof-lemma8-eq-2} implies that
	$$\|\tilde{x}^0 - \tilde{x}^1\|_2\geq \frac{\lambda_s}{2L}\left(\frac{\lambda_s}{2L} + \|x_\star\|_2\right)^{-1}\Delta(\tilde{x}^0).$$
This proves \eqref{lemma8-conclusion-1}. To prove \eqref{lemma8-conclusion-2}, without loss of generality, we assume that $\Delta(\tilde{x}^0) = \|\tilde{x}^0-x_\star\|_2$. Using \eqref{rel:gen_weak1} and the fact that $F_t(\tilde{x}^0;\tilde{x}^0) = F(\tilde{x}^0)$ and noticing that $\tilde{x}^1 = S_t(\tilde{x}^0)$ is the minimizer of $F_t(\cdot;\tilde{x}^0)$, we have
	\begin{align*}
		&F_t(\tilde{x}^0;\tilde{x}^0) - F_t(\tilde{x}^1;\tilde{x}^0)\geq F(\tilde{x}^0) - F_t(x_\star;\tilde{x}^0)\\
		&= { F(\tilde{x}^0) - F(x_\star;x^0) - \frac{L}{2}\Delta^2(\tilde{x}^0)}\\
  &{\geq F(\tilde{x}^0) - F(x_\star) - L\Delta^2(\tilde{x}^0)} \\
  &\geq \lambda_s\Delta(\tilde{x}^0) - L\Delta^2(\tilde{x}^0),
	\end{align*}
 where the {second inequality is due to \eqref{rel:gen_weak1}}, and the last inequality is due to Assumption \ref{ass:sharpness}. 
	Using the fact that \tcb{$\Delta(\tilde{x}^0)\leq (F(\tilde{x}^0) - F(x_\star))/\lambda_s \leq \lambda_s/(2L)$} based on Assumption \ref{ass:sharpness}, we have
	$$F_t(\tilde{x}^0;\tilde{x}^0) - F_t(\tilde{x}^1;\tilde{x}^0)\geq (\lambda_s - {L}\frac{\lambda_s}{2L})\Delta(\tilde{x}^0)={\frac{\lambda_s}{2}}\Delta(\tilde{x}^0),$$
    which proves \eqref{lemma8-conclusion-2}. Therefore, the proof of Lemma \ref{thm:sharp_diff2} is complete.
\end{proof}

Now, we are ready to give an upper bound for $J_k$. 
\begin{lemma}\label{subiters_nesterov} 
\tcb{Assume that Assumption \ref{ass:sharpness} holds}. For Algorithm \ref{Nesterov} under either options \eqref{low2} \tcb{with $\rho_l>0$} or \eqref{high2} \tcb{with $\rho_h\in (0,1/4)$}, for any $k\in\mathbb{N}$, if \eqref{lemma-14-equation-1} holds, 
  the following statements hold.
	\begin{itemize}
		\item[(a)] (Low Accuracy) When using option \eqref{low2}, we have
  \begin{align}\label{lemma12-conclude-a}
   J_k\leq \left\lceil\max\left\{\frac{M_1\|B_k\|_2^2}{\Delta(x^k)},C'\right\}\right\rceil-1.
   \end{align}
		\item[(b)] (High Accuracy) When using option \eqref{high2}, we have
		  \begin{align}\label{lemma12-conclude-b}
			J_k\leq \left\lceil\max\left\{\frac{M_2\|B_k\|_2^2}{\Delta^2(x^k)},C'\right\}\right\rceil-1.
		\end{align}
	\end{itemize}
	Here, $C'$ and $C''$ are the constants in Lemma \ref{lemma:primal_recovery} and $M = \frac{\lambda_s}{2L}\left(\frac{\lambda_s}{2L} + \|x_\star\|_2\right)^{-1}$, $M_1 = {2}C''tm(\rho_l+1)/(\lambda_s\rho_l)$ and $M_2 = 4C''t^2m(\rho_h+1)/(\rho_hM^2).$
\end{lemma}

\begin{proof} [Proof of Lemma \ref{subiters_nesterov} ]
(a). 
By choosing 
\begin{equation}\label{lemma12-proof-a-1}
    j = \left\lceil\max\left\{C',\left(\frac{{2}C''tm\|B_k\|_2^2(\rho_l+1)}{\rho_l\lambda_s\Delta(x^k)}\right)\right\}\right\rceil-1,
\end{equation}
we have
\begin{equation}\label{lemma12-proof-a-2}
    \frac{C''tm\|B_k\|_2^2}{j+2}\leq \lambda_s\rho_l\Delta(x^k)/({2}+{2}\rho_l),
\end{equation}
which, together with \eqref{lemma10-conclud}, yields that 
\begin{align}
    & H_k(z_k(\lambda_a^{j+1})) - D_k(\lambda_a^{j+1}) \nonumber\\ \leq & \lambda_s\rho_l\Delta(x^k)/({2}+{2}\rho_l) \nonumber\\
    \leq & \frac{\rho_l}{1+\rho_l}\left(H_k(0) - \min_{z\in\mathbb{R}^n} H_k(z)\right),\label{lemma12-proof-a-3}
\end{align} 
where the last inequality is from \eqref{lemma8-conclusion-2}. From \eqref{lemma12-proof-a-3} we have
\begin{align*}
& H_k(z_k(\lambda_a^{j+1})) - D_k(\lambda_a^{j+1}) \\
\leq & \rho_l\left(-H_k(z_k(\lambda_a^{j+1})) +D_k(\lambda_a^{j+1}) + H_k(0) - \min_{z\in\mathbb{R}^n} H_k(z)\right) \\
\leq & \rho_l\left(H_k(0) - H_k(z_k(\lambda_a^{j+1}))\right),
\end{align*}
where the last inequality is due to \eqref{weak-duality}, and this gives \eqref{low2}. Therefore, \eqref{low2} should have already been satisfied when \eqref{lemma12-proof-a-1} holds. This proves \eqref{lemma12-conclude-a} in Lemma \ref{subiters_nesterov}(a).


(b). 
By choosing 

\begin{equation}\label{lemma12-proof-b-1}
    j = \left\lceil\max\left\{C',\left(\frac{4C''t^2m\|B_k\|_2^2(\rho_h+1)}{M^2\rho_h\Delta^2(x^k)}\right)\right\}\right\rceil-1,
\end{equation}
we have
\begin{equation}\label{lemma12-proof-b-2}
    \frac{C''tm\|B_k\|_2^2}{j+2}\leq M^2\rho_h\Delta^2(x^k)/\left(4t(1+\rho_h)\right),
\end{equation}
which, together with \eqref{lemma10-conclud}, yields that 
\begin{align}
    & H_k(z_k(\lambda_a^{j+1})) - D_k(\lambda_a^{j+1}) \nonumber\\ \leq & M^2\rho_h\Delta^2(x^k)/\left(4t(1+\rho_h)\right) \nonumber\\
    \leq & \rho_h\|x^k-S_t(x^k)\|_2^2/\left(4t(1+\rho_h)\right),\label{lemma12-proof-b-3}
\end{align} 
where the last inequality is from \eqref{lemma8-conclusion-1}. Denote $z_{k\star} = \argmin_{z\in\mathbb{R}^n} H_k(z)$, which indicates that $\|z_{k\star}\|_2 = \|x^k - S_t(x^k)\|_2$. We have
$$\frac{1}{2t}\|z_k(\lambda_a^{j+1}) - z_{k\star}\|_2^2\leq H_k(z_k(\lambda_a^{j+1})) - D_k(\lambda_a^{j+1}),$$
which follows from the fact that $H_k(\cdot)$ is $\frac{1}{t}$-strongly convex and \eqref{weak-duality}. From \eqref{lemma12-proof-b-3}, we have $\|z_k(\lambda_a^{j+1}) - z_{k\star}\|_2^2\leq \rho_h/\left(2(1+\rho_h)\right)\|z_{k\star}\|_2^2.$ So, by the Cauchy-Schwarz inequality, we have
\begin{align*}
	\frac{\rho_h}{2t}\|z_k(\lambda_a^{j+1})\|_2^2 & \geq \frac{\rho_h}{4t}\|z_{k\star}\|_2^2- \frac{\rho_h}{2t}\|z_k(\lambda_a^{j+1}) - z_{k\star}\|_2^2\\
	&\geq \rho_h\|z_{k\star}\|_2^2/\left(4t(1+\rho_h)\right).
\end{align*}
Together with \eqref{lemma12-proof-b-3}, \eqref{high2} should have already been satisfied when \eqref{lemma12-proof-b-1} holds. This proves \eqref{lemma12-conclude-b} in Lemma \ref{subiters_nesterov}(b). 

Therefore, the proof of Lemma \ref{subiters_nesterov} is complete. 
\end{proof}

Next, we are ready to present the proof of Theorem \ref{overall_iplnesterov}.

\begin{proof}[Proof of Theorem \ref{overall_iplnesterov}] We first prove part (a) of Theorem \ref{overall_iplnesterov} and then prove part (b) of Theorem \ref{overall_iplnesterov} in what follows.

(a).
Lemma \ref{bound_b2} and \eqref{lemma-14-equation-2} indicate that $\sup_{k\in\mathbb{N}}\|B_k\|_2\leq B(\lambda_s/(2L))$.
Together with \eqref{lemma-14-equation-1}, \eqref{lemma12-conclude-a}, \eqref{lemma-14-equation-2}, we obtain 
  \begin{equation}\label{proof-thm3-eq-2}
      J_k\leq M_1B^2(\lambda_s/(2L))/\Delta(x^k).
  \end{equation} 
		
We now prove the desired result. If $\epsilon\geq \Delta(x^0)$, then $J(\epsilon) = 0$ and \eqref{thm3-conclusion-a} holds. If $\epsilon < \Delta(x^0)$, we denote $K_\epsilon$ as the smallest index such that $\Delta(x^{k})>\epsilon,\forall k\leq K_\epsilon$ and $\Delta(x^{K_\epsilon+1})\leq\epsilon$. 
 From Lemma \ref{Lip_F} we have
 \begin{align}\label{thm3-proof-inequality-2}
	&L(\|x_\star\|_2+\lambda_s/(2L))\Delta(x^k) \\ \geq & F(x^k)-F(x_\star)\nonumber\\
\geq & \left(F(x^{K_\epsilon}) - F(x_\star)\right)\left(\frac{2+4\rho_l}{1+4\rho_l}\right)^{K_\epsilon - k}\nonumber\\
\geq &  \lambda_s\Delta(x^{K_\epsilon})\left(\frac{2+4\rho_l}{1+4\rho_l}\right)^{K_\epsilon - k} \nonumber\\ \geq & \epsilon\lambda_s\left(\frac{2+4\rho_l}{1+4\rho_l}\right)^{K_\epsilon - k}, \ \forall\ 0\leq k\leq K_\epsilon.\nonumber
\end{align}
Here we explain how these inequalities were obtained. 
 The first inequality follows from Lemma \ref{Lip_F} and \eqref{lemma-14-equation-2}. Since \eqref{lemma-14-equation-1} holds, \eqref{lemma5-a-eq-2} holds for any $k\in\mathbb{N}$ and it implies the second inequality in \eqref{thm3-proof-inequality-2}.  The third inequality in \eqref{thm3-proof-inequality-2} follows from Assumption \ref{ass:sharpness} and the last inequality in \eqref{thm3-proof-inequality-2} follows from the fact that $\Delta(x^{K_\epsilon})>\epsilon$.
 
From \eqref{thm3-proof-inequality-2} we know that $\forall\ 0\leq k\leq K_\epsilon$, it holds that
\begin{equation}\label{proof-thm3-eq-3}
    \Delta(x^k)\geq \frac{\epsilon\lambda_s}{L(\|x_\star\|_2+\lambda_s/(2L))}\left(\frac{2+4\rho_l}{1+4\rho_l}\right)^{K_\epsilon - k}.
\end{equation}
Therefore, we have
		\begin{align*}
			J(\epsilon) &= \sum_{k=0}^{K_\epsilon} J_k\\
   &\leq \sum_{k=0}^{K_\epsilon} M_1B^2(\lambda_s/(2L))/\Delta(x^k)\\
			&\leq \sum_{k=0}^{K_\epsilon} \frac{M_1B^2(\frac{\lambda_s}{2L})L(\|x_\star\|_2+\lambda_s/(2L))}{\lambda_s\epsilon}\left(\frac{2+4\rho_l}{1+4\rho_l}\right)^{k-K_\epsilon}\\
			& = \frac{E_2}{\epsilon(2+4\rho)}\sum_{k=0}^{K_\epsilon}\left(\frac{2+4\rho_l}{1+4\rho_l}\right)^{k-K_\epsilon}\leq \frac{E_2}{\epsilon}.
		\end{align*}
  where the first inequality follows from \eqref{proof-thm3-eq-2} and the second inequality follows from \eqref{proof-thm3-eq-3}. 
  This completes the proof of Theorem \ref{overall_iplnesterov}(a).  

(b). 
Lemma \ref{bound_b2} and \eqref{lemma-14-equation-4} indicates that $\sup_{k\in\mathbb{N}}\|B_k\|_2\leq B(E_0/2)$.
Together with \eqref{lemma-14-equation-1}, \eqref{lemma-14-equation-5}, \eqref{lemma12-conclude-b}, ,
  \begin{equation}\label{proof-thm3-eq-J-bound}
      J_k\leq M_2B^2(E_0/2)/\Delta^2(x^k),\forall k\geq 0,
  \end{equation}
which follows from our choice of $E_3\leq B(E_0/2)\sqrt{M_2/C'}$ such that
  $M_2B^2(E_0/2)/\Delta^2(x^k)\geq M_2B^2(E_0/2)/\Delta^2(x^0)\geq C'$. Next, we prove the conclusion. We adopt the same definition of $K_\epsilon$ as in (a) and pick $\{\epsilon_i\}_{i=1}^\infty = \{\Delta(x^i)\}_{i=1}^\infty$ so that $K_{\epsilon_i} = i-1$. From \eqref{proof-lemma14-eq-1} we have 
  \begin{equation}\label{proof-thm3-eq-6}
      \Delta(x^{i-1})\geq \sqrt{E_0\epsilon_i}, \forall i\in\mathbb{N}.
  \end{equation}
  From \eqref{lemma-14-equation-5} and \eqref{proof-thm3-eq-6} we have
  \begin{equation}\label{proof-thm3-eq-7}
      \Delta(x^k)\geq 2^{i-1-k}\sqrt{E_0\epsilon_i},\forall\ 0\leq k\leq i-1.
  \end{equation}
  Finally, for any $i\in\mathbb{N}$, we have 
		\begin{align*}
			&J(\epsilon_i) = \sum_{k=0}^{i-1} J_k\leq \sum_{k=0}^{i-1}\frac{M_2B^2(E_0/2)}{4^{i-1-k}E_0\epsilon_i}\\
			&\leq \frac{4M_2B^2(E_0/2)}{3E_0\epsilon_i} = \frac{E_4}{\epsilon_i}.
		\end{align*}
  Here, the first inequality follows from \eqref{proof-thm3-eq-J-bound} and \eqref{proof-thm3-eq-7}, and the second inequality follows from the fact that
  $$\sum_{k=0}^{i-1}\frac{1}{4^{i-1-k}}\leq \sum_{k=-\infty}^{i-1}\frac{1}{4^{i-1-k}} = \frac{4}{3}.$$
  This completes the proof of Theorem \ref{overall_iplnesterov}(b). 
  
  Therefore, the proof of Theorem \ref{overall_iplnesterov} is complete.
\end{proof}

\section{Conclusion}\label{sec:conclusion}

In this paper, we proposed a new inexact proximal linear algorithm for solving the robust phase retrieval problem. Our contribution lies in the two adaptive stopping criteria for the subproblem in the proximal linear algorithm. We showed that the iteration complexity of our inexact proximal linear algorithm is in the same order as the exact proximal linear algorithm. Under the sharpness condition, we are able to prove the local convergence of the proposed method. Moreover, we discussed how to use the FISTA for solving the subproblem in our inexact proximal linear algorithm, and analyzed the total oracle complexity for obtaining an $\epsilon$-optimal solution under the sharpness condition. Numerical results demonstrated the superior performance of the proposed methods over some existing methods. 


\bibliographystyle{IEEEtran}
\bibliography{inexact_pl_refs}
\end{document}